\documentclass[12pt,a4paper]{article}

\usepackage{amsmath}
\usepackage{amsthm}
\usepackage{amssymb}
\usepackage{graphicx}
\usepackage{color}
\usepackage{nicefrac}
\usepackage{titlesec}
\usepackage{needspace}

\usepackage{tikz}
\usetikzlibrary{calc}

\titleformat{\section}[block]{\needspace{8\baselineskip}\bigskip\centering\large\bfseries}{\thesection}{1em}{\vspace{3ex}}
\titleformat{\subsection}[block]{\needspace{4\baselineskip}\centering\normalsize \bfseries}{\thesubsection}{1em}{\vspace{2ex}}
\titleformat{\subsubsection}[block]{\normalfont\em}{\thesubsubsection.}{2em}{\hspace{2em}}

\usepackage[hidelinks]{hyperref}

\newcommand{\co}{{\mathbb C}}

\newcommand{\re}{{\mathbb R}}
\newcommand{\n}{{\mathbb N}}

\newcommand{\cA}{{\mathcal{A}}}

\newcommand{\cT}{{\mathcal{T}}}

\newcommand{\cD}{{\mathcal{D}}}

\newcommand{\bx}{{\boldsymbol{x}}}
\newcommand{\by}{{\boldsymbol{y}}}

\newcommand{\be}{{\boldsymbol{e}}}
\newcommand{\bh}{{\boldsymbol{h}}}
\newcommand{\bp}{{\boldsymbol{p}}}

\newcommand{\ba}{{\boldsymbol{a}}}
\newcommand{\bb}{{\boldsymbol{b}}}
\newcommand{\bc}{{\boldsymbol{c}}}
\newcommand{\bv}{{\boldsymbol{v}}}
\newcommand{\bu}{{\boldsymbol{u}}}
\newcommand{\bw}{{\boldsymbol{w}}}

\newcommand{\norm}[1]{\left\lVert#1\right\rVert}
\newcommand{\vardot}{\mathord{\,\cdot\,}} 

\newtheorem{thm}{Theorem}[section]
\numberwithin{thm}{section}
\newtheorem{prop}[thm]{Proposition}
\newtheorem{lma}[thm]{Lemma}
\newtheorem{cor}[thm]{Corollary}
\newtheorem{rmk}[thm]{Remark}
\newtheorem{ex}[thm]{Example}
\newtheorem{defi}[thm]{Definition}
\newtheorem{conj}[thm]{Conjecture}

\newtheorem{innercustomgeneric}{\customgenericname}
\providecommand{\customgenericname}{}
\newcommand{\newcustomtheorem}[2]{%
  \newenvironment{#1}[1]
  {%
   \renewcommand\customgenericname{#2}%
   \renewcommand\theinnercustomgeneric{##1}%
   \innercustomgeneric
  }
  {\endinnercustomgeneric}
}

\newcustomtheorem{customproblem}{Problem}

\date{}

\author{Thomas~Mejstrik%
\thanks{University of Vienna, Austria {e-mail: \tt\small  thomas.mejstrik@gmx.at}}
\ and
Vladimir~Yu.~Protasov%
\thanks{DISIM, University of L'Aquila,  {e-mail: \tt\small v-protassov@yandex.ru}}}

\title{Elliptic polytopes and invariant norms \\ of linear operators%
\thanks{The first author 
is sponsored by the Austrian Science Foundation (FWF) grant P~33352.
The second author is supported by the RFBR grants 19-04-01227
and 20-01-00469}}

\begin{document}



\maketitle 

\begin{abstract}

We address the problem of constructing elliptic polytopes in~$\re^d$, which are  
convex hulls of finitely many two-dimensional ellipses with a common center. 
Such sets  arise in the study of 
spectral properties of matrices, 
asymptotics of   long matrix products, 
in the Lyapunov stability, etc. 
The main issue in the construction is to decide 
whether a given ellipse is in the convex hull of others.
The computational complexity of this 
problem is analysed by considering an equivalent optimisation problem. 
We show that the number of local extrema of that problem may grow exponentially in~$d$. 
For $d=2, 3$, it admits an explicit solution for an arbitrary number of ellipses; 
for higher dimensions,  several geometric methods for approximate solutions are derived. 
Those methods are analysed  numerically and 
their efficiency is demonstrated in applications.
\medskip 


\begin{flushright}
\noindent \textbf{AMS 2020} {\em Mathematical Subject
classification: 
52A21, 
39A30, 
15A60, 
90C90
}
\end{flushright}

\end{abstract}

\section{Introduction}\label{sec_introduction}  

Convex hulls of two-dimensional ellipses in~$\re^d$ 
are applied in the evaluation  of Lyapunov functions and  of extremal norms of linear 
operators, in the study of stability of discrete-time linear systems and in the 
computation of the joint spectral radius. The construction of such convex hulls is computationally hard, especially in high dimensions.  
It is reduced to the following question: to decide whether a given ellipse 
is contained in the convex hull of other given ellipses.  
We study the complexity and suggest several methods of its approximate solution. 

Note that a solution merely by approximating each ellipse with a polygon is extremely 
inefficient and is hardly realisable if we want a good precision. 
That is why the problem requires other approaches based on various geometric ideas.  
The paper is concluded with  numerical results and  applications.

\begin{defi}\label{d.10}
An elliptic polytope in~$\re^d$ is a convex hull of several two-dimensional 
ellipses centred at the origin. 
Those ellipses which are not in the convex hull of the others 
are called vertices of the elliptic polytope.  
\end{defi}
An ellipse can be degenerate, in which case it is 
a segment centred at the origin. So, every (usual) polytope symmetric about the origin is 
also an elliptic polytope. We usually define an  ellipse either by a pair of vectors 
$\ba, \bb \in \re^d$ as the set of points 
$\ba \cos t + \bb \sin t\, , \, t \in \re$ and denote it as~$E(\ba, \bb)$,
or by a complex vector~$\bv = \ba \,+\, i\, \bb \in \co^d$, 
and denote it as~$E(\bv)\, = \, E(\ba, \bb)$,
where $\ba = {\rm Re}\, \bv, \, \bb = {\rm Im}\, \bv$ are real and imaginary parts of~$\bv$, respectively.

Two complex vectors~$\bv_1, \bv_2\in\co^d$ 
define the same ellipse if either~$\bv_2 = z\, \bv_1$ or $\bv_2 = z\, \bar \bv_1$
for some~$\, z\in \co, \, |z| = 1$.

\subsection{Motivation}
\label{sec_motivation}  

Construction of elliptic polytopes arise naturally in the study of 
spectral properties of matrices, 
of asymptotics of long matrix products, 
the stability of linear dynamical systems, and in  related problems. 
Below we consider some of these applications. 

\subsubsection*{Application 1.~Norms in $\co^d$ restricted to~$\re^d$}

It is well-known that every convex body in~$\re^d$ symmetric about the 
origin generates a norm in~$\re^d$,
called its Minkowski norm. 
In contrast, not every convex body in~$\co^d$ (identified with~$\re^{2d}$) 
defines a norm in~$\co^d$. Such bodies have a  particular structure:  
if~$S$ is a unit sphere of a norm~$\norm{\vardot}$ in~$\co^d$, 
then for every $\bv \in S$ the curve $\{e^{-i s } \bv \, | \, s \in \re\}$
lies on~$S$. Indeed, $\norm{e^{-i s } \bv} = \norm{\bv} = 1$.  
Note that the  real  part of the point $\bv(s)$
runs over the ellipse~$E(\bv)$ as $s\in \re$. 
Thus, a unit ball of a norm in~$\re^d$ induced by an arbitrary complex norm is a convex hull 
of a (possibly infinite) set of ellipses. 
In particular, a piecewise linear approximation of a norm in the complex space 
is a {\em balanced complex polytope} (see Definition~\ref{d.20})
with real and imaginary parts being elliptic polytopes. Therefore, elliptic polytopes
are the real and imaginary part of a polyhedral approximation for unit balls of norms in~$\co^d$.

\subsubsection*{Application 2.~Lyapunov functions for linear dynamical systems}

For a discrete time linear system of the form $\bx(k+1) = A\bx(k), \, k \ge 0$, 
where $A$ is a constant $d\times d$ matrix, an important issue is a construction of a 
Lyapunov function~$f(\bx)$, for which $f(A\bx) \le \, \lambda\, f(\bx), \, \bx \in \re^d$.
  If such a function exists for $\lambda = 1$, then the system is stable, if it exists for 
  $\lambda < 1$, then it is asymptotically stable. 
  A Lyapunov function provides  
  a detailed information on the asymptotic behaviour of the trajectories $\bx(k)$
  as $k\to \infty$. A standard approach is to find a quadratic Lyapunov function 
  $f(\bx) = \sqrt{\bx^T M \bx}$, where $M$ is a positive definite matrix. 
  By the Lyapunov theorem such a matrix~$M$ exists whenever $\rho(A) < 1$, where $\rho$
  is the spectral radius (maximal modulus of eigenvalues).  
  The quadratic Lyapunov function  can be found either by solving 
  a semidefinite programming problem $A^TMA \prec M$
  or by finding all eigenvectors of~$A$ (for the sake of simplicity we assume that 
  $A$ does not have multiple eigenvalues). In high dimensions, however,  
  both those methods become hard. In this case one should 
  consider  Lyapunov functions from other classes, for example, 
  from the class of polyhedral functions. To construct a polyhedral Lyapunov function
  one needs to to find a polytope~$P$ such that $AP \subset P$. 
  Such a polytope can be constructed iteratively starting with an arbitrary 
  polytope $P_0$ and running the process $P_{k+1} = {\rm co}\, \{AP_k,  P_k\}$,
  where ${\rm co}$ denotes the convex hull.
  When $P_{n+1} = P_n$ the algorithm halts and we set $P = P_n$.  
   However, if $\rho(A)$  is close to one,
   then the number of vertices of~$P_n$ may become very large. 
   This can be avoided by  including the leading eigenvector~$\bv$ of~$A$ 
   in the set of vertices of~$P_0$. 
   If  $\bv$ is not real, then $P_0$ is replaced by an elliptic polytope: 
   a convex hull of~$E(\bv)$ with several other vertices. 
   In this case, all~$P_k$ and the final polytope~$P$ will be elliptic. 
   Thus, the iterative algorithm 
   with elliptic polytopes constructs the desired Lyapunov function.    

\subsubsection*{Application 3.~Computation of the joint spectral radius}
 
  This is one of the most important applications of elliptic 
  polytopes. 
  The joint spectral radius of matrices is the maximal rate of 
  asymptotic growth of norms of their long products.
  For an arbitrary family  $\cA = \{A_1, \ldots , A_m\}$ of $d\times d$ matrices,
  the joint spectral radius (JSR) is the limit
 \begin{equation}
 \rho(\cA) \ = \ \lim_{k \to \infty}\ \max_{A(j)\in \cA} \ \big\|A(k)\ldots A(1)\big\|^{1/k}\, .
 \end{equation}
 Originated with J.\,K.~Rota and G.~Strang in 1960 
 the joint spectral radius found numerous applications, 
 see~\cite{CCS2005}\cite{GP13}\cite{Jung2009} for surveys.  The computation of the joint spectral radius, 
 even approximate, is a hard problem. 
 The \emph{Invariant polytope algorithm} 
 introduced in~\cite{GP13} makes it possible to find a precise value 
 of $\rho(\cA)$ for a vast majority of matrix families. 
 Its idea traces back to~\cite{GWZ}\cite{P96}. Recent works~\cite{GP16}\cite{P17}\cite{Mej2020}
 develop updated versions of that algorithm which efficiently perform computations 
 in dimensions up to $d=25$ for arbitrary matrices and up to several thousands 
 for nonnegative matrices. 
 The main idea of the Invariant polytope algorithm is the following: 
 First, we find  a candidate for the spectrum maximizing product $\Pi$ 
 of matrices from~$\cA$, for which the value 
 $\lambda = \rho(\Pi)^{1/|\Pi|}$  is maximal, where $|\Pi|$ denotes the length of
 the product~$\Pi$. 
 We make an assumption that the leading eigenvalue of~$\Pi$ is unique and simple.  
 Then we construct an 
 {\em extremal norm} $\norm{\vardot}$ in $\re^d$ such that  
 $\norm{A\bx} \le \lambda \norm{\bx}$ for all~$A\in \cA$, $\bx \in \re^d$. 
 Once such a norm is found, we have proved that~$\rho(\cA) = \lambda$. 
 If $\lambda \in \re$, then the  extremal norm is  constructed iteratively, 
 starting with the leading 
 eigenvector~$\bv$ 
 of~$\Pi$, considering its~$m$ images~$\lambda^{-1}A\bv , \ A\in \cA$, 
 then constructing their~$m^2$ images, etc..
 To avoid the exponential growth of the number of points we  
 remove all redundant points (those in the convex hull of others) in each iteration. 
 If this process halts after several iterations (no new points appear), then 
 the convex hull of the collected points forms an {\em invariant polytope}~$P$, 
 for which $AP \subset \lambda P$ for all $A\in \cA$. 
 The Minkowski norm of this polytope is extremal. 
 If the leading eigenvalue of~$\Pi$
 is nonreal, then $P$ can be found as a balanced complex polytope (Definition~\ref{d.20})
 by the same iterative procedure
 starting with complex leading eigenvector~$\bv$. 
 This approach was elaborated in~\cite{GWZ}\cite{GZ07}\cite{GZ09} 
 and showed its efficiency for the JSR computation.
 There were, however, some disadvantages. First of all, 
 this method was mostly applied in low dimensions. 
 Second, for matrix families with complex leading eigenvalue
 this algorithm suffers, since the uniqueness of the leading eigenvalue assumption 
 is violated. 
 Indeed, in the latter case the complex conjugate number 
 $\bar \lambda $ is also a leading eigenvalue. 
 
 We modify this method by using elliptic polytopes instead 
 of balanced complex polytopes.  
 The process starts with the ellipse
 $E(\bv)$, then in each iteration we 
 compute  the images of the previous ellipses, and remove the 
 redundant ones. 
 Thus, in the case of a complex leading eigenvalue, 
 the joint spectral radius can be found by the 
 iterative construction of an invariant elliptic polytope. 
 Usually the elliptic polytope has much less number of vertices (ellipses) 
 than the balanced complex polytope, 
 which allows to speed up its convergence and to apply it in higher dimensions. 
 See Section~\ref{sec_applications} for details. 
 
\subsubsection*{Application 4. Stability of linear switching systems}
 
The extremal norm~$\norm{\vardot}$ constructed above 
by an elliptic polytope plays not only an auxiliary role 
 for computing the joint spectral radius.  
 It is also an independent interest as a Lyapunov function 
 for the discrete time {\em linear switching system} 
$\bx(k+1) = A(k)\bx(k), \,  A(k) \in \cA, \, k \ge 0$, 
see~\cite{B88}\cite{GL15}\cite{G95}\cite{K07}\cite{PW08}\cite{W02}.  
 For this system, $\rho(\cA)$ has the meaning of the Lyapunov exponent, and the extremal 
 norm~$\norm{\vardot}$ is a Lyapunov function. 
 Thus, the Lyapunov function of a discrete time linear switching system 
 is constructed as the Minkowski functional of an elliptic polytope.

\subsection{The statement of the problem}\label{sec_problemstatement}  
  
The following problem, which will be referred to as Problem~\ref{problem_ee}
(\emph{ellipse in ellipses}) is crucial in constructing and studying  elliptic polytopes. 
\begin{customproblem}{EE}  
\label{problem_ee}
For given ellipses~$E_0, \ldots , E_N$ in~$\re^d$,
decide whether $\, E_0 \subset {\rm co}\, \{E_1, \ldots , E_N\}$.  
\end{customproblem}

An efficient solution to Problem~\ref{problem_ee} makes it possible 
to ``clean'' every set of ellipses 
removing  redundant ones and  leaving only the vertices of an 
elliptic polytope containing all others. 
All the  aforementioned  applications in Section~\ref{sec_introduction} 
are based on the use of Problem~\ref{problem_ee}. 

Concerning Application~1, an arbitrary norm in~$\co^d$ can be approximated  by a polyhedral 
norm~$\norm{\bx}= \max_{j}|(\bv_j, \bx)|$. Consider the restriction of this norm to~$\re^d$. 
Solving Problem~\ref{problem_ee} 
one removes redundant vectors~$\bv_k$.
A term~$|(\bv_k, \bx)|$ is redundant if and only if the ellipse~$E(\bv_k)$ 
is contained in the convex hull of the others~$E(\bv_j)$, $j\ne k$.

In the other applications,
solving Problem~\ref{problem_ee} also plays a major role. 
In the iterative construction of the Lyapunov functions 
and in the Invariant polytope algorithms, the 
removal  of redundant ellipses in each iteration prevents 
the exponential growth of the number of ellipses and actually 
makes those algorithms applicable. Moreover, reducing the number of ellipses makes 
the Lyapunov function simpler and more convenient for applications.  

Our second topic is the algorithmically implementation of the solution of Problem~\ref{problem_ee}.
In particular, we aim to modify the algorithm of the JSR computation (Application~3) 
by using elliptic polytopes instead of  complex polytopes. 
The same construction will be applied for finding invariant Lyapunov  
functions for switching systems (Application~4).

\subsection{Possible approaches}\label{sec_approacher}  

An analogue to Problem~\ref{problem_ee} for usual
polytopes is solved by the standard linear programming technique.  
For elliptic polytopes, we are not aware of any known method. 
To the best of our knowledge, the only problem considered in the literature, 
which is related to Problem~\ref{problem_ee}, 
is the construction of a  balanced complex polytope. 
This technique was developed in~\cite{GP13}\cite{GWZ}\cite{GZ07}\cite{GZ09}\cite{GZ15} 
for finding extremal Lyapunov functions in~$\co^d$ 
and for computing the joint spectral radius.
It is based on the following fact: 
 An ellipse~$E(\bv_0)$ 
is contained 
in the convex hull~ ${\rm co}\, \{E(\bv_1), \ldots , E(\bv_N)\}$
if there exist complex numbers $z_k$ such that $\bv_0 = \sum_{k=1}^m z_k \bv_k$
and $\sum_{k=1}^m  |z_k| \le 1$.  
This condition, however, is only sufficient but not necessary. 
Moreover, 
it turns out that for solving Problem~\ref{problem_ee}, this method
gives a rather rough approximate solution. 
We are going to show that its approximation  
factor is~$\nicefrac{1}{2}$ and this estimate is tight (Theorems~\ref{th.20} and
\ref{th.30} in Section~\ref{sec_approximatesolutions}). Moreover, it works only if 
we add the complex conjugate vectors $\bar \bv_k$ to the set of vectors~$\bv_k$, 
otherwise the approximation factor is zero.
I.e.~we will not obtain even an approximate solution. 
This aspect has been missed in the recent literature on the joint spectral radius computation. 

Natural  questions arise -- How to get a precise solution of Problem~\ref{problem_ee}
and what is the complexity of this problem?
What could be done to obtain approximate solutions with better approximation factors?
Having answered those questions one can speed up the Invariant polytope algorithm 
for the joint spectral radius computation,
construct extremal Lyapunov functions for discrete time systems 
that would be  easier to define and to compute than those presented in the literature, 
and address other  applications.

\subsection{The main results and the structure of the paper}\label{sec_structure}  

In Section~\ref{sec_preliminary} we give necessary definitions, 
notation, and formulate auxiliary facts.  
In Section~\ref{sec_equivalentproblem} we rewrite Problem~\ref{problem_ee} 
in the optimisation form and study its complexity.  
The problem is highly nonconvex and, therefore, can be hard. 
Indeed, we show that it is not simper than the problem of maximising 
a quadratic form of rank 2 over a centrally symmetric polyhedron. 
We conjecture that the latter problem is NP-hard. 
An argument for that is established in~Theorem~\ref{th.8},
 a positive semidefinite quadratic form of rank 2 in~$\re^k$ 
under~$O(k)$ linear constraints may have $2^k$ points of local maxima. 

In Section~\ref{sec_lowdimensions} we show that in low dimensions 
Problem~\ref{problem_ee} admits precise solutions. 
In general, if the dimension is fixed, then the problem has a polynomial 
(in the number of ellipses) solution, although hardly realizable for $d\ge 4$. 
For higher dimensions we can deal with approximate solutions only 
(Section~\ref{sec_approximatesolutions}). 

In Section~\ref{sec_methodB} we analyse the {\em complex polytope method} 
for solving Problem~\ref{problem_ee}. 
Its idea is close to those originated with Guglielmi, Zennaro, and Wirth~\cite{GWZ}\cite{GZ07}\cite{GZ09}. 
By this method we  reduce Problem~\ref{problem_ee}
to a conic programming problem.
First we observe one aspect missed in the literature: 
This method does not work, unless we add complex conjugate vectors to all given vectors
(Proposition~\ref{p.17}). 
After this slight modification, 
the method becomes applicable and gives an approximate solution to Problem~\ref{problem_ee} 
with an approximation factor of at least $\nicefrac{1}{2}$. 
This is shown in~Theorem~\ref{th.20}. 
This factor, in general,  cannot be improved as shown in Theorem~\ref{th.30}.
Certainly, for some initial data the approximation can be sharper. 
However, the empirical estimate obtained for random elliptic polytopes 
gives the expected value of the approximation factor around~$\nicefrac{1}{\sqrt{2}}$,
 which is also quite rough. 
The corresponding numerical results are presented in Section~\ref{sec_numerics}.

Then, in Section~\ref{sec_methodD} we derive another approach, 
which allows us to obtain approximate solutions 
with an arbitrary approximation factor (the factor~$1$ corresponds to the precise solution). 
This is a corner cutting algorithm, which reaches a very high accuracy. 
By solving~$k$ conic programming problems 
 with $N$ constraints, where $N$ is the number of ellipses, 
 we get an approximate solution with an approximate factor 
 of~$ 1-\nicefrac{\pi^2}{2(k+1)^2}$. 
Already for $k=3$, we obtain  the factor at least 
$\nicefrac{\sqrt{2}}{2} \simeq 0.707$, which is better than by the polytope method. 
For $k = 5$, 
the factor is approximately $0.923$. These are the ``worst case estimates'' and in practice 
the corner cutting  algorithm reaches a much higher accuracy already for small~$k$. 

In Section~\ref{sec_methodE} we consider a modification 
of the corner cutting algorithm to a linear programming
problem. To this end we apply the idea of Ben-Tal and Nemirovski
of approximating quadrics by projections of higher dimensional polyhedra. 
This gives a fast algorithm of approximation of ellipses by projections of polyhedra. 
Combining  this construction with the corner cutting method significantly improves 
the accuracy.

After numerical results presented in Section~\ref{sec_numerics} we demonstrate 
some applications. 
We show that the elaborated methods of solving Problem~\ref{problem_ee} allow us to
efficiently construct Lyapunov functions for linear dynamical systems 
even in high dimensions, for which a tradition way of finding 
a quadratic Lyapunov function by s.d.p. is hardly reachable. 
For the linear switching systems, our results speed up 
the Invariant polytope algorithm in case of nonreal leading eigenvalue and 
 reduce a lot the number of ellipses defining the extremal Lyapunov function of the system.

\section{Preliminary facts and notation}\label{sec_preliminary}  

Throughout the paper we denote vectors by bold letters and numbers by standard letters. 
Thus $\bx = (x_1, \ldots , x_d)^T \in \re^d$. As usual, for two complex vectors
$\bv, \bu \in \co^d$, their scalar product is $(\bv, \bu) = \sum_{k=1}^d v_k \bar u_k$. 
For two real vectors $\ba, \bb$, we consider the ellipse 
$$
E \ = \  E(\ba, \bb) \ = \ \bigl\{\, \ba \cos t \, + \, \bb \sin t \ | \  t \in \re \bigr\}.
$$
This is an ellipse with conjugate radii (the halfs of conjugate diameters) $\ba, \bb$. 
For a complex vector $\bv = \ba + i\bb$ with real~$\ba$ and~$\bb$, we write 
$E(\bv)$.  
For every~$s \in \re$, the real and complex parts of the vector 
$e^{- i\, s}\bv \, = \, \ba_s - i\bb_s$  are conjugate 	directions of the same ellipse, 
and therefore $E(\ba_s, \bb_s) = E(\ba, \bb)$ for all~$s \in \re$. Indeed, 
$$
e^{- i\, s}\bv\ =\ (\cos s \, - \, i \sin s)\, (\ba + i\bb)\ =\ 
(\ba \cos s \,  + \, \bb \sin s )\ +  \ i\, (- \ba \sin s \,  + \, \bb \cos s ), 
$$
hence, $\ba_s \, = \, \ba \cos s \,  + \, \bb \sin s $ and 
$\bb_s = \ba \sin s \,  - \, \bb \cos s$. 
Therefore, $\ba_s \cos t \, + \, \bb_s \sin t\, =\, \ba \cos (t+s)\, + \, 
\bb \sin (t_s)$. The pair $(\ba_s, \bb_s)$ is the image of~$(\ba, \bb)$ 
after the elliptic rotation by the angle~$s$ along the ellipse~$ E = E(\ba, \bb)$.  
Consequently,  the vectors $\ba_s, \bb_s$ are also conjugate directions 
of the ellipse~$E$. 

Elliptic polytopes are real parts of the so-called  balanced complex polytopes defined as 
follows:   
\begin{defi}\label{d.20}
A balanced convex hull of a set~$K \subset \co^d$
is 
$$
{\rm cob} \, (K)\quad = \quad 
\Bigl\{\ \sum_{k=1}^{n}  \  z_k \bv_k \quad \Bigl|    
\quad z_k \in \co\, , \ \bv_k \in K  \, , \ \sum_{k=1}^{n}  |z_k| \le 1 , 
\quad n \in \n \, \Bigr\}. 
$$
A balanced convex set is a subset of~$\co^d$ that coincides with its balanced convex hull. 
A balanced convex hull of a finite set of points is a balanced complex polytope. 
\end{defi}
If $G$ is a balanced complex polytope, then ${\rm Re}\, G$
is a convex hull of ellipses. 
Indeed, if $G \, = \, 
{\rm cob} \, \{\bv_1, \ldots , \bv_N\}\,$  and $\, \ba_k = {\rm Re}\, \bv_k, \, 
\bb_k = {\rm Im}\, \bv_k$, $\ k = 1, \ldots , N$, 
then for arbitrary complex numbers $z_k = r_k e^{-it_k},$ 
where $r_k = |z_k|$,  $\, k = 1, \ldots , N$, we have 
$$
{\rm Re}\, z_k \bv_k \ = \  r_k \bigl( \ba_k \cos t_k \, + \, \bb_k \sin t_k   \bigr) \, 
$$
and hence, the set ${\rm Re}\, G$ consists precisely of the points 
$\sum_k r_k \bu_k $ with $\, \bu_k \in E(\bv_k)$ and $\sum_k r_k \le 1$. 
This is $\, {\rm co}\, \{E(\bv_1), \ldots , E(\bv_N)\}$. 

Note that the balanced polytopes $G$ and $\bar G = \{\bar \bv \ | \ \bv \in G\}$ 
have the same real parts and hence,  
generate the same elliptic polytope~$P$. 
Therefore, $P$ does not change  if we replace 
$G$ by ${\rm cob}\, \{G, \bar G\}$. 
In what follows, if the converse is not stated, 
 we assume that the balanced complex polytope is symmetric with respect to the 
conjugacy, i.e.\ $G = \bar G$. Clearly, this holds if  so is the set of vertices $\{\bv_k\}$. 

\begin{rmk}\label{r.10}
{\em The imaginary part of a balanced complex polytope is the same 
elliptic polytope~$P$. Indeed, 
\begin{equation*}
\begin{aligned}
{\rm Im}\, z_k \bv_k \ &= \  
r_k \Bigl( - \ba_k \sin t_k  \, + \, \bb_k \cos t_k   \Bigr)\ = \ 
r_k \Bigl(\, \ba_k \cos \bigl(t_k+ \frac{\pi}{2}\bigr)  \, 
     + \, \bb_k \sin \bigl(t_k+ \frac{\pi}{2}\bigr)  \ \Bigr)\ \\ 
& = \ -i \, {\rm Re}\, z_k \bv_k\, .
\end{aligned}
\end{equation*}
We see that the set ${\rm Im}\, G$ consists of the points 
$\sum_k r_k \bv_k $ with $\, \bv_k \in E_k$ and $\sum_k r_k \le 1$, and thus, 
${\rm Im}\, G = P = {\rm Re}\, G$. 
Of course, the same is true for an arbitrary balanced convex set: 
its  real and imaginary parts  coincide. 
}
\end{rmk}

\section{Equivalent  optimisation problems and their complexity}
\label{sec_equivalentproblem}  

To analyse the complexity and possible solutions 
of  Problem~\ref{problem_ee} we reformulate it as an optimisation problem.

\subsection{Reformulation of Problem~\ref{problem_ee}}\label{sec_reformulation}  

Let $P = {\rm co}\, \{E_1, \ldots , E_N\}$ be an elliptic polytope. 
An ellipsoid  $E_0$ is not contained  in~$P$ 
if and only if $P$ possesses a hyperplane of support that 
intersects~$E_0$ at two points. For the outward normal vector~$\bx$
of that hyperplane, we have  
 $$
 \sup_{\bw_0 \in E(\ba_0, \bb_0)}\, (\bx , \bw_0) \quad > \quad 
  \sup_{\bw \in P}\, (\bx , \bw). 
 $$
 Note that
 $$
 \sup_{\bw \in E(\ba, \bb)}\, (\bx , \bw)\quad  = \quad 
 \sup_{t \in \re}\, (\bx , \ba)\, \cos t \, + \,  (\bx , \bb)\, \sin t
 \quad  = \quad 
 \sqrt{ (\bx , \ba)^2 \, + \,  (\bx , \bb)^2}\, . 
 $$
 Therefore, 
\begin{equation*}
\begin{aligned}
&{\sup_{\bw_0 \in E(\ba_0, \bb_0)}}\, (\bx , \bw_0)  =
\sqrt{ (\bx , \ba_0)^2 \, + \,  (\bx , \bb_0)^2}\\
&{\sup_{\bw \in P}}\, (\bx , \bw) =
\max_{k=1, \ldots , n} \sqrt{ (\bx , \ba_k)^2 \, + \,  (\bx , \bb_k)^2}\, . 
\end{aligned}
\end{equation*}
Thus, the assertion $E_0 \not \subset P$
is equivalent to the existence of  a solution $\bx \in \re^d$
for the system of inequalities 
\begin{equation}\label{eq.eq0}
(\bx , \ba_0)^2 \, + \,  (\bx , \bb_0)^2 \quad > 
\quad (\bx , \ba_k)^2 \, + \,  (\bx , \bb_k)^2\, ,
\qquad k=1, \ldots , N\, . 
\end{equation}
Normalising the vector~$\bx$, it can be assumed that 
$(\bx , \ba_0)^2 +   (\bx , \bb_0)^2  = 1- \varepsilon$ where $\varepsilon > 0$ 
is a small number, in which case
the system~\eqref{eq.eq0} is equivalent to the system 
$(\bx , \ba_k)^2  +  (\bx , \bb_k)^2 \le 1\, , \ k=1, \ldots , N$. 
Thus, we have proved: 
\begin{thm}\label{th.5}
Problem~\ref{problem_ee} is equivalent to the following optimisation problem: 
\begin{equation}\label{eq.prob1}
\left\{
\begin{aligned}
(\bx , \ba_0)^2 \, + \,  (\bx , \bb_0)^2 \ \to & \ \max\\
(\bx , \ba_k)^2 \, + \,  (\bx , \bb_k)^2\, \le & \, 1, \qquad  k=1, \ldots , N\, ,  
\end{aligned}
\right. 
\end{equation}
with $d$ variables $(x_1, \ldots , x_d)^T = \bx$ and given 
vectors~$\ba_k, \bb_k  \in \re^d$.  
\end{thm}
Therefore, we need to maximize a positive semidefinite quadratic form of rank two on the 
intersection of cylinders.

\subsection{The complexity of Problem~\ref{problem_ee}}  
\label{sec_complexityee}

Maximisation of a convex function over a convex set is usually nontrivial. 
Problem~\ref{problem_ee} and its reformulation~\eqref{eq.prob1}, 
does not seem to be an exception. Moreover, 
the feasible domain is defined by $N$ quadratic inequalities in~$\re^d$ and 
does not look simple either. Geometrically this is an intersection of~$N$
elliptic cylinders in~$\re^d$ with two-dimensional bases. 
The following result sheds some 
light on the complexity of this problem and hence, 
on the complexity of~Problem~\ref{problem_ee}.

\begin{thm}\label{th.7}
Maximization of a positive semidefinite quadratic form of rank two over 
a centrally symmetric polyhedron defined by $2N$ linear inequalities in~$\re^d$
can be
reduced to Problem~\ref{problem_ee}.  
  \end{thm}
\begin{proof}
An origin-symmetric polyhedron is defined 
by $N$ inequalities ${(\bx, \ba_k)^2 \le 1}$. 
Choosing arbitrary  numbers $t_1, \ldots , 
 t_N \in \bigl(0, \frac{\pi}{2} \bigr)$, we set 
 $\ba_k = \bh_k\cos t_k , \  \bb_k = \bh_k\sin t_k $. Then the polytope 
 is defined by the system of constraints of the reformulation~\eqref{eq.prob1}. 
 Finally,  every quadratic form of rank two can be written as 
 $(\bx , \ba_0)^2 +  (\bx , \bb_0)^2$ for suitable $\ba_0, \bb_0$, which completes the 
 proof.    
\end{proof}

\begin{conj}\label{con.10}
Maximising  a positive semidefinite quadratic form of rank two over 
a centrally symmetric polyhedron is NP-hard. 
\end{conj}

Let us recall that the problem of maximizing a  positive semidefinite quadratic 
form over a polyhedron is 
 NP-hard  even if that polyhedron is a unit cube, since it is not simpler then the 
 Max-Cut  problem~\cite{GW95}\cite{H99}. 
 Moreover, even its approximate solution is NP-hard~\cite{H01}.  
 However, the rank two assumption 
may significantly simplify it. For example, the complexity of this problem on the unit cube 
becomes not only polynomial, but linear with respect to~$Nd$. It is reduced to 
finding the diameter of a flat zonotope, 
see~\cite{FFL05} for more result on this and related problems.  Nevertheless, we
believe in the high complexity of this problem.   
One argument for that is a large number of local extrema. 
The following theorem states that if we drop the assumption of the symmetry of the 
polyhedron, then the number of local maxima with different values of the function 
can be exponential.

\begin{thm}\label{th.8}
For each~$N\ge 2$ there exists 
 a polyhedron in~$\re^N$ with less than~$2N$ facets 
 and a positive semidefinite quadratic form of rank two on that polyhedron which has at least 
$2^{N-2}$ points of local maxima with different values of the function.  
  \end{thm}
The proof is in the Appendix. 

\bigskip

On the other hand, as we shall see in the next section, 
in low dimensions, Problem~\ref{problem_ee} admits efficient solutions. 

\section{Problem~\ref{problem_ee} in low dimensions}  
\label{sec_lowdimensions}

In dimensions~$d=2, 3$, Problem~\ref{problem_ee} can be efficiently solved.
The solution in the two-dim\-en\-sional case is simple, the three-dimensional case 
is computationally harder. 

\subsection{%
\texorpdfstring{The dimension $d=2$}{The dimension d=2}
}  

In the two-dimensional plane the solvability of the system~\eqref{eq.eq0} 
is explicitly  decidable, 
which solves Problem~\ref{problem_ee}.
\begin{prop}\label{p.10}
In the case $d=2$, Problem~\ref{problem_ee} 
admits an  explicit solution for arbitrary ellipses 
$E_0, \ldots , E_N$. 
The complexity of the solution is linear in~$N$.   
\end{prop}
The proof is constructive and gives the method for the solution. 
\begin{proof}
Denote $\bx = (x, y)^T$ and rewrite the inequalities~\eqref{eq.eq0}
 in coordinates. After  simplifications we get $A_ky^2 +  2B_kxy +  C_kx^2 > 0$, 
 $k = 1, \ldots , N$, where $A_k, B_k, C_k$ are known coefficients. 
 The set of solutions to the  $k^{\rm th}$ inequality is  
 $\frac{y}{x} \in I_k$, where $I_k$ is either the interval with ends at 
 the roots of the quadratic equation $A_k t^2 + 2B_kt + C_k = 0$, 
if $A_k < 0$ (if there are no real roots, then $I_k = \emptyset$); 
or the union of two open rays with the same roots, 
if $A_k > 0$ (if there are no real roots, then $I_k = \re$); 
or one ray if $A_k = 0, B_k \ne 0$; the other cases are trivial. 
Then the solution of the system~\eqref{eq.eq0} 
consists of points~$\bx = (x, y)^T$ such that the ratio $ \frac{y}{x}$ belongs
to the intersection   
$\,  \bigcap\limits_{k=1, \ldots , N}\, I_k$. 
Hence, $E_0 
\subset {\rm co}\, \{E_1, \ldots , E_N\}$ if and only if this intersection is empty, 
i.e.\ $\bigcap\limits_{k=1, \ldots , N}\, I_k\, = \, \emptyset$.
\end{proof}

\subsection{%
\texorpdfstring{The dimension $d=3$}{The dimension d=3}
}  

In the three-dimensional space the solvability of the 
system~\eqref{eq.eq0} is also explicitly  decidable, but much harder than in dimension~$2$. 
\begin{prop}\label{p.15}
In the case $d=3$, Problem~\ref{problem_ee},  for arbitrary ellipses 
$E_0, \ldots , E_N$, is reduced to solving of $O(N^2)$ bivariate quadratic systems of equations. 
\end{prop}
\begin{proof}
 Denote $\bx = (x, y, z)^T$ and rewrite the inequalities~\eqref{eq.eq0}
 in coordinates. This is  a system of homogeneous inequalities of degree~$2$. 
 After the division by~$z^2$,  we get a system of 
 quadratic inequalities  $f_i(x, y) < 0, \ i = 1, \ldots , N$. 
 It is compatible precisely when so is the system 
 $f_i(x, y) - \varepsilon \le  0, \ i = 1, \ldots , N$, 
 for some small~$\varepsilon > 0$.
 Denote by~$\cD$ the set of its solutions 
 and assume it is nonempty. 
 This is a closed subset of~$\re^2$ bounded by arcs of the  quadrics 
 $\Gamma_i = \bigl\{(x, y)^T \in \re^2 \ | \  f_i(x,y) - \varepsilon  = 0\bigr\}$. 
 The closest to the origin point of~$\cD$ belongs to one of the three sets: 
 1) the origin itself; 
 2)  points of pairwise  intersections $\Gamma_i \cap \Gamma_j, \, i\ne j$; 
 3) closest to the origin  points of $\Gamma_i\, , \, i = 1, \ldots , N$. 
 If some of those quadrics coincide or are circles centred at the origin, 
 then we reduce the number of quadrics by the standard argument. 
 Otherwise, the set 2 contains at most $4\cdot \frac{N(N-1)}{2} \, = \, 
 2N(N-1)$ points; the set 3 contains at most $4N$ points. Hence, 
 if $\cD$ is nonempty, then it contains one of the points of the sets 1, 2, 3.
 Thus, to decide if $\cD$ is empty or not, one needs to 
 take each of those $2N(N-1) + 4N +1 \, = \, 2N^2 + 2N +1$ points 
 and check whether it belongs to~$\cD$, 
 i.e.\ satisfies  all the 
 inequalities $f_i(x, y) - \varepsilon\le 0, \ i = 1, \ldots , N$. 
 If the answer is affirmative for at least one point,
 then system~\eqref{eq.eq0} is compatible and $E_0 \ne \subset P$,
 otherwise $E_0 \ne \subset P$. 
 
 Evaluating each of those $O(N^2)$  points, except for the first one,  is done by solving a system 
 of two quadratic inequalities.  
\end{proof}

\subsection{Problem~\ref{problem_ee} in a fixed dimension}  

Similarly to Proposition~\ref{p.15},
one can show that Problem~\ref{problem_ee} 
in~$\re^d$ is reduced to 
$O(N^{d-1})$ systems of $d$ quadratic equations with $d$ variables. 
The complexity of this problem is formally polynomial in~$N$, with the degree depending on~$d$. 
However the method used in the case $d=3$ 
(the exhaustion of points of intersections and of points minimizing the distance to the 
origin) becomes non-practical for higher dimensions.

\section{Approximate solutions}  
\label{sec_approximatesolutions}

Apart from the low-dimensional cases, most likely, no 
efficient algorithms exist to obtain an explicit solution of Problem~\ref{problem_ee}. 
That is why 
we are interested in  approximate solutions with a given relative error 
(approximation factor) according to the following definition: 
\begin{defi}\label{d.30}
A method solves Problem~\ref{problem_ee} approximately with a factor~$q \in [0,1]$ if 
 it decides between two cases: 
either  $E_0  \not \subset P $ or  $ \, q E_0  \subset P$. 
 	
\end{defi}
So, the extreme case~$q=1$ corresponds to a precise solution, 
the other extreme case~$q=0$ means that the method does not give any approximate solution. 
 We consider two methods. The first one is based on the construction of 
 a balanced complex polytope. Such polytopes were  deeply analysed in~\cite{GWZ}\cite{GZ09}\cite{GZ15}. 
 At the first site, the results of those works give a straightforward 
 solution to Problem~\ref{problem_ee}. However, this is not the case.  We are going to show that 
 the balanced complex polytope method provides only 
 an approximate solution with the factor~$q=\nicefrac{1}{2}$ and this value cannot be improved. 
Moreover, this approximation is attained only after a slight modification of this method%
 , otherwise the approximation factor may drop to zero. 
 Then we introduce the second method which provides a better approximation
 (with the factor~$q$ arbitrarily close to~$1$, i.e.\ to the precise solution).

\section{The complex polytope method}  
\label{sec_methodB}

We have an elliptic polytope $P = {\rm co}\, \{E_1, \ldots , E_N\}$ 
and an ellipse~$E_0$ and 
need to decide whether or not  $E_0 \subset P$.  
For each ellipse~$E_k$, we choose arbitrary conjugate radii~$\ba_k, \bb_k$, 
thus $E_k = E_k(\ba_k, \bb_k), \ k = 1, \ldots , N$. 
Define $\bv_k = \ba_k + i \bb_k $ and consider the balanced complex polytope 
\begin{equation}
G \ = \ {\rm cob}\, \Bigl\{\ \bv_k \ \Bigl| \ k = 1, \ldots N\, \Bigr\}\, .
\end{equation}
To get an approximate solution of Problem~\ref{problem_ee} 
we consider the following auxiliary problem: 

\begin{customproblem}{EE*}  
\label{problem_ees}
For 
points~$\bv_0, \ldots , \bv_N \in \co^d$,
decide whether or not  $\, \bv_0 \in {\rm cob}\, \{\bv_1, \ldots , \bv_N\}$.   
\end{customproblem}

What is the relation between Problem~\ref{problem_ee} and~\ref{problem_ees}? 
Clearly, if $\bv_0 \in G$, then $E_0 \subset P$. Indeed, if $\bv_0 \in G$, then 
$e^{it}\bv_0 \in G$ for all~$t \in \re$, hence 
$E_0 = {\rm Re}\, \{e^{it}\bv_0 , \ t \in \re\} \, \subset \, 
{\rm Re}\, G$.  However, the converse is, in general, not true
and  Problem~\ref{problem_ees} is not equivalent to Problem~\ref{problem_ee}. Moreover, 
Problem~\ref{problem_ees} does not even provide  an approximate solution to 
Problem~\ref{problem_ee} with a positive factor. This means that 
the assertion $E_0 \subset P$ does not imply 
the existence of a positive~$q$ such that  $q\bv_0 \in G$. 
 \begin{prop}\label{p.17}
Problem~\ref{problem_ees} gives an approximate solution 
to Problem~\ref{problem_ee} with the factor~$q=0$. 
\end{prop}
\begin{proof}
Let $\ba_0 = (1,0)^T$ and $\bb_0 = (0,1)^T$, and 
 $\, \ba_1 = \bb_0, \bb_1 = \ba_0$.
Clearly, $E_0$ and $E_1$ both coincide with the unit disc, hence 
$P$ is also  a unit disc and $E_0 \subset P$. 
On the other hand, no positive number $q$ exists such that 
$q\bv_0 \in G$. Indeed, $G = \{z \bv_1 \ \big| \ |z| \le 1\}$. 
Denote $z = t + i u$. If $q\bv_0 = z\bv_1$, 
then 
$$
q\bv_0 \ = \ t \ba_1 \, - \, u \bb_1 \, +\,  
i\, (t \bb_1 \, + \, u \ba_1) \ = \ 
t \bb_0 \, - \, u \ba_0 \, +\,  
i\, (t \ba_0\, + \,  u \bb_0)\, .
$$
Hence, 
$$
q \ba_0 \ = \ t \bb_0 \, - \, u \ba_0 \quad\text{and}\quad
q\bb_0 \ = \ t \ba_0\, + \,  u \bb_0,
$$
which is coordinatewise $(q, 0) = (-t, u)$
and $(0, q) = (t, u)$.
Therefore, $q = t = u = 0$. 
\end{proof}
 
Thus, Problem~\ref{problem_ees} does not give an approximate solution 
to Problem~\ref{problem_ee}. Nevertheless, 
under an extra assumption that  $G$ is self-conjugate, 
it does provide an approximate solution with the factor~$\nicefrac{1}{2}$. 
This factor is tight and cannot be improved. 
This follows from Theorems~\ref{th.20} and~\ref{th.30}  proved below.
Before formulating them, we briefly discuss the practical issue. 

To solve Problem~\ref{problem_ees} 
we  consider a self-conjugate balanced complex polytope 
$G = {\rm cob}\, \{\bv_k ,\allowbreak \bar{\bv}_k \ | \ k = 1, \ldots N\}$. 
As we noted in Remark~\ref{r.10}, it has the same real part~$P$ as 
the balanced polytope~$G = {\rm cob}\, \{\bv_k \ | \ k = 1, \ldots , N\}$. 
Problem~\ref{problem_ees} is solved for~$G$ by the following optimisation problem: 
\begin{equation}\label{eq.cp}
\left\{
\begin{aligned}
&t_0\ \to \ \max,\quad \mbox{subject to:}\\
&\sqrt{t_{j}^2 + u_j^2} \le r_j, \ j = 1, \ldots , 2N\\
&\sum_{j=1}^{2N} r_j \, \le \, 1\\
&t_0\ba_0 \, = \, \sum_{k=1}^{N} \bigl(t_{k} \ba_k \, - \, u_{k} \bb_k\bigr)\, + \, 
\bigl(t_{k+N} \ba_k \, + \, u_{k+N} \bb_k\bigr)\\
&t_0\bb_0 \, = \, \sum_{i=1}^{\ell} \bigl(u_{k} \ba_k \, + \, t_{k} \bb_k\bigr)\, + \, 
\bigl(u_{k+N} \ba_k \, - \, t_{k+N} \bb_k\bigr)
\end{aligned}
\right. 
\end{equation}
This problem  finds the biggest $t_0$ such that 
 $t_0\bv_0$ is a balanced complex combination 
of the points $\bv_1, \ldots , \bv_N, \allowbreak \bar \bv_1, \ldots , \bar \bv_N$.
 The coefficients of this combination are $z_k = t_k + i u_k$, $k = 1, \ldots , 2N$, 
the points $\bv_k, \bar \bv_k$ correspond to the 
coefficients $z_k, z_{k+N}$ respectively. 
This is a convex conic programming problem with variables $t_0, t_k, u_k$, 
where $k = 1, \ldots , 2N$.  
It is solved by the interior point method 
on Lorentz cones.
If $t_0 \ge  1$, then $\bv_0 \in G$ and vice versa. 

In Section~\ref{sec_numerics} we demonstrate the numerical results showing that 
the problem is efficiently solved in relatively low 
dimension 2 to 25 and for the number of ellipses up to 1000.

\bigskip

Now we are going to see that  if $\bv_0 \notin G$, then $ \, E_0 \not \subset \frac12 P$.
Dividing by two, 
we obtain an approximate solution to Problem~\ref{problem_ee} with the factor at least~$\nicefrac{1}{2}$: 
 if $\frac12 \bv_0 \notin G$, 
 then $ \, E_0 \not \subset   P$, otherwise, if $\frac12 \bv_0 \in G$, then 
 $\frac12 E_0 \subset P$. 
 
\begin{thm}\label{th.20}
A precise solution of Problem~\ref{problem_ees}
 gives an approximate solution to Problem~\ref{problem_ee} with the factor~$q\ge \frac12$. 
\end{thm}

\begin{proof}
It suffices to show that if $\bv_0 \notin G$, 
then $ \, E_0 \not \subset \frac12 P$. If a point $\bv_0 = \ba_0 + i \bb_0$ 
does not belong to $G$, then it can be separated from $G$ by a nonzero 
functional $\bc = \bx+ i\by$, which means 
$$
{\rm Re}\, (\bc, \bv_0) \ > \ \sup_{\bv \in G} \, {\rm Re}\, (\bc, \bv)\, .
$$
Rewriting the scalar product in the left-hand side we obtain  
$$
 (\bx, \ba_0) \, - \, (\by, \bb_0)\ > \ \sup_{\bv \in G} \ {\rm Re}\,  (\bc, \bv) \, .
$$
Note that $e^{-it}\bv \in G$ for all~$t \in \re$. 
Substituting this for $\bv$ in the right-hand side, we get 
$$
 (\bx, \ba_0) \, - \, (\by, \bb_0)\ > 
 \ \sup_{\bv \in G, \, t \in \re} \ {\rm Re}\,  (\bc, e^{-it}\bv) \, .
$$
Since ${\rm Re}\,  (\bc, e^{-it}\bv) \, = 
\, \big((\bx, \ba) - (\by, \bb)\big)\cos t\, + \, 
\big((\bx, \bb) + (\by, \ba)\big)\sin t $
and the supremum of this value over all $t \in \re$ is equal to 
$$
\sqrt{\, \big((\bx, \ba) - (\by, \bb)\big)^2\ + \ \big((\bx, \bb) + (\by, \ba)\big)^2},
$$
 we conclude that 
\begin{equation}\label{eq.eq1}
 (\bx, \ba_0) \, - \, (\by, \bb_0)\ > 
 \ \sup_{\ba + i \bb \in G} 
 \ \sqrt{\big((\bx, \ba) - (\by, \bb)\big)^2\, + \, \big((\bx, \bb) + (\by, \ba)\big)^2}.
\end{equation}
Since $G$ is symmetric with respect to the 
conjugacy, we have  $\bar \bv \in G$ and hence  
$i \bar \bv \, = \, \bb + i\ba\, \in \, G$. 
Hence, one can interchange $\ba$ and $\bb$ in~\eqref{eq.eq1} and get 
\begin{equation}\label{eq.eq2}
 (\bx, \ba_0) \, - \, (\by, \bb_0)\ > 
 \ \sup_{\ba + i \bb \in G} 
 \ \sqrt{\big((\bx, \bb) - (\by, \ba)\big)^2\, + \, \big((\bx, \ba) + (\by, \bb)\big)^2}
\end{equation}
If $\ - \, (\bx, \ba) \cdot (\by, \bb)\, + \, (\bx, \bb) \cdot (\by, \ba) \, \ge 0$,
then\eqref{eq.eq1} yields 

\begin{equation}\label{eq.eq3}
 (\bx, \ba_0) \, - \, (\by, \bb_0)\ > 
 \ \sup_{\ba + i \bb \in G} 
 \ \sqrt{(\bx, \ba)^2 + (\by, \bb)^2\, + \, (\bx, \bb)^2 + (\by, \ba)^2}.
\end{equation}
Otherwise, if $\ - (\bx, \ba) \cdot (\by, \bb)\, + \, (\bx, \bb) \cdot (\by, \ba) \le 0$, then we apply~\eqref{eq.eq2} 
and arrive at the same inequality~\eqref{eq.eq3}. 
Since inequality~\eqref{eq.eq3} is strict, we take squares of its both parts 
and obtain that there exists $\varepsilon > 0$ such that 
\begin{equation}\label{eq.eq4}
 \big( (\bx, \ba_0) \, - \, (\by, \bb_0) \big)^2\
>
\ \sup_{\ba + i \bb \in G} \ 
\big((\bx, \ba)^2 + (\by, \bb)^2\, + \, (\bx, \bb)^2 + (\by, \ba)^2\big) 
\ + \ \varepsilon\, . 
\end{equation}
Denote by $\bp$ the vector from the set~$\{\bx, \by\}$
on which the maximum 
$$
\max_{\bp \in  \{\bx, \by\}} \ (\bp, \ba_0)^2 \, + \, (\bp, \bb_0)^2
$$
is attained. Note that $\bp$ depends on~$\bc$ and~$\bv_0$ only. 
Hence, for every 
point $\bv  = \ba + i \bb\in G$, we have 
$$
(\bp, \ba)^2 \, +  \, (\bp, \bb)^2 \ \le \ 
(\bx, \ba)^2 \, +  \, (\bx, \bb)^2\, + \, (\by, \ba)^2 \, + \, (\by, \bb)^2 \ \le \  
 \big( (\bx, \ba_0) \, - \, (\by, \bb_0)\big)^2 \, - \, \varepsilon 
$$
 On the other hand, 
 \begin{align*}
 \big( (\bx, \ba_0) \, - \, (\by, \bb_0)\big)^2  & \le  2 \, (\bx, \ba_0)^2 \, + \, 
 2(\by, \bb_0)^2  \ \\
 &\le
 2 \, \big( (\bx, \ba_0)^2 \, + \, (\bx, \bb_0)^2 \, + \, 
 (\by, \ba_0)^2 \, + \,(\by, \bb_0)\big)  \\
 &\le  
4 \, \big( (\bp, \ba_0)^2 \, + \, (\bp, \bb_0)^2 \, \big).
\end{align*}
Thus, 
$$
 (\bp, \ba_0)^2 \, + \, (\bp, \bb_0)^2 \, - \, \frac{\varepsilon}{4} \quad \ge \quad  
 \frac14 \, \big( (\bp, \ba)^2 \, + \, (\bp, \bb)^2 \, \big) \, , 
$$ 
 and consequently, 
 $$
 \sqrt{(\bp, \ba_0)^2 \, + \, (\bp, \bb_0)^2 \, - \, \frac{\varepsilon}{4}} \quad \ge \quad  
 \frac12 \, \sqrt{ (\bp, \ba)^2 \, + \, (\bp, \bb)^2 \, } \, , 
$$ 
 Now observe that the right-hand side of this inequality is 
 equal to $\sup_{\bw \in E(\ba, \bb)}\, (\bp , \bw)$
 and the 
 left-hand side is smaller than  $\sup_{\bw_0 \in E(\ba_0, \bb_0)}\, (\bp , \bw_0)$. 
 Therefore, for every pair~$\ba, \bb \in \re^d$ such that 
$\ba + i \bb\in G$, we have 
 $$
 \sup_{\bw_0 \in E(\ba_0, \bb_0)}\, (\bp , \bw_0) \quad > \quad 
 \frac{1}{2}\, \sup_{\bw \in E(\ba, \bb)}\, (\bp , \bw)\, . 
 $$
 This means that there exists a point 
 $\widehat \bw \in E(\ba_0, \bb_0)$ such that 
 $(\bp , \widehat \bw) > \frac12\, \sup_{\bw \in E(\ba, \bb)}\, (\bp , \bw)$. 
 This holds for every point $\ba + i \bb \in G$, 
 in particular, for each point $\ba_k + i \bb_k, \ k = 1, \ldots , N $.
 Hence, the linear functional $\bp$ strictly separates 
 the point $\hat \bw$ of the ellipsoid $E_0$ from 
 all ellipsoids $\frac12\, E_k$, 
 i.e.\ from their  convex hull. 
Therefore,
$\widehat \bw \notin \frac12 P$ and hence $E_0 \not \subset \frac12 P$. 
\end{proof}

After Theorem~\ref{th.20} the natural question arises 
whether the approximation factor~$\nicefrac{1}{2}$
can be increased.
The following theorem shows that the answer is negative. 
 \begin{thm}\label{th.30}
The factor $q = \frac{1}{2}$ in Theorem~\ref{th.20} is sharp.  
\end{thm}
\begin{proof}
It suffices  to give an example where this factor 
can be arbitrarily close to $\nicefrac{1}{2}$. 
Consider the set $S$ of pairs of vectors 
$(\ba, \bb) \in \re^2\times \re^2$ such that 
$\ba , \bb $
 are collinear  and  
 $|\ba|^2 + |\bb|^2 \le 1$. Then define 
$Q = \{\ba + i \bb \ | \ (\ba , \bb) \in S\}$. Thus, $Q \subset \co^2$. 

Since each  pair $(\ba , 0)$ with $|\ba| = 1$ belongs to~$S$, we see that 
the set  ${\rm Re}\, Q$ contains a unit disc centred at the origin. 
In our notation this disc can be denoted as
$E(\be_1, \be_2)$, where $\be_1 = (1,0)^T$ and $
\be_2 = (0,1)^T$. Furthermore, if $\bv \in Q$, then 
$\bar \bv \in Q$ and $e^{it}\bv \in Q$ for each $t \in \re$.
The first assertion is obvious, to prove the second one we 
observe that $e^{i\tau}\bv \, = \, \ba_{\tau} + i \bb_{\tau}$ with  
$\ba_{\tau} = \ba \cos \tau - \bb \sin \tau$ and  
$\bb_{\tau} = \ba \sin \tau + \bb \cos \tau$. 
Clearly, $\ba_{\tau}$ and $\bb_{\tau}$ are collinear and 
$|\ba_{\tau}|^2 + |\bb_{\tau}|^2 \, = \, 
|\ba|^2 + |\bb|^2 \, \le  \, 1$. 
Every point of the balanced convex hull 
$G \, = \, {\rm cob}\, (Q)$ has the form 
$\sum_{k = 1}^N z_k \bu_k\, = \,  \sum_{k = 1}^N |z_k|e^{i\tau_k}\bu_k$, where 
$z_k = |z_k| e^{i\tau_k}$ and $\bu_k \in Q, \, \sum_{k=1}^N|z_k| \le 1$. Writing $t_k = |z_k|$
and $\bv_k = e^{i\tau_k}\bu_k$ and using that $\bv_k \in Q$, 
we see that every point of $G$ has the form 
$\sum_{k = 1}^N t_k \bv_k$ with all $\bv_k$ from~$Q$
and $\sum_{k=1}^N t_k \le  1$. 

Now  let us solve Problem~\ref{problem_ees} for the set~$G$ and for the 
vector~$t_0(\be_1 + i \be_2)$. We find the maximal positive~$t$
for which this vector belongs to~$G$. 
 We have  $t_0(\be_1 + i \be_2) \, = \, \sum_{k=1}^N t_k \bv_k$
 with $\bv_k = \ba_k + i \bb_k \, \in \, Q$ and $t_k \ge 0$, $\sum_{k=1}^N t_k \le 1$.
 We are going to show that $t_0 \le \frac12$. 

Let $\ba_k$ be co-directed to the vector~$(\cos \gamma_k, \sin \gamma_k)^T$;
the vector $\bb_k$ has the direction $\varepsilon_k\, (\cos \gamma_k, \sin \gamma_k)^T$, 
where $\varepsilon_k \in \{1, -1\}$. 
Since $|\ba_k|^2 + |\bb_k|^2 \le 1$, 
it follows that there is an angle $\delta_k = \bigl[0, \frac{\pi}{2}\bigr]$
and a number $h_k \in [0,1]$ such that $|\ba_k| = h_k \cos \delta_k \, , \ 
|\bb_k| = h_k \sin \delta_k $. We have $\sum_{k=1}^N t_k \ba_k = t_0 \be_1$. 
In the projection to the abscissa, we have $\sum_{k=1}^N t_k (\ba_k, e_1) = t_0 $
and hence, 
$$
\sum_{k=1}^N t_k h_k \cos \gamma_k \, \cos \delta_k \,  = \, t_0 \, . 
$$
Similarly, after the projection of the equality $\sum_{k=1}^N t_k \bb_k = t_0 \be_2$
to the vector~$\be_2$, we get  
$$
\sum_{k=1}^N \, \varepsilon_k t_k h_k  \sin \gamma_k \, \sin \delta_k \,  = \, t_0 \, . 
$$
Taking the sum of these two equalities, we obtain 
 $$
 \sum_{k=1}^N t_k h_k \cos (\gamma_k - \varepsilon_k \delta_k) \,  = \, 2t_0 \, . 
 $$
Since all numbers $h_k \cos (\gamma_k - \varepsilon_k \delta_k)$ 
do not exceed one, we conclude that 
$$
\sum_{k=1}^N t_k \,  \ge  \, 2t_0 \, , 
$$
and therefore, $t_0 \le \frac12$. Hence, 
for the unit disc $E_0(\be_1, \be_2)$ 
and the set of ellipses $\{\, E(\ba, \bb) \ | \ \ba + i\bb  \in G \}$, 
the solution of Problem~\ref{problem_ees} 
gives the approximation for Problem~\ref{problem_ee} with the factor at most 
$\frac12$. This is not the end yet, 
since $Q$ is infinite and so $G$ is not a balanced complex 
polytope.  However,  $G$ can be approximated by   a 
 balanced polytope with an arbitrary precision. For the obtained 
 balanced polytope, 
the approximation factor is close to $\frac12$. 
Since it can be made arbitrarily close, the 
proof is completed. 
\end{proof}

\section{The corner cutting method}  
\label{sec_methodD}

A straightforward approach to approximate solution of Problem~\ref{problem_ee} 
could be  to replace $E_0$ 
by a sufficiently close circumscribed polygon and then to decide whether 
all its vertices belong to~$P$. 
However, this idea turns out to be not efficient: 
to provide a good approximation factor this polygon 
will have many vertices and hence the algorithm will work slowly.  
We derive another approach based on step-by-step relaxation 
by cutting angles of a polygon. This procedure localizes 
the most distant point of $E_0$ from $P$ and checks whether that point belongs to~$P$. 
We begin with the following auxiliary problem~\ref{problem_pe} (\emph{point in ellipses}), 
which can be seen as a special case of Problem~\ref{problem_ee}

\begin{customproblem}{PE}  
\label{problem_pe}
In the space~$\re^d$ there are ellipses $E_1, \ldots , E_N$
and a point~$\bw$. Find $\|\bw\|_P$, where $P = {\rm co}\, \{E_1, \ldots , E_N\}$. 
\end{customproblem}

In particular, deciding whether $\|\bw\|_P \le 1$ is equivalent to 
a special case of Problem~\ref{problem_ee} 
when the ellipse~$E_0$ degenerates to a segment 
 $[-\bw, \bw]$. This problem can be efficiently solved.
Either precisely, by the 
conic programming method in subsection~\ref{subsec_methodD-conic},
or approximately by the linear programming method 
presented in Section~\ref{sec_methodE}.

\subsection{The algorithm of corner cutting}  

We begin with description of the main idea and then define a routine of the algorithm. 

\subsubsection*{The idea of the algorithm.}
It may be assumed that $E_0$ is a unit circle. We construct a sequence of 
polygons circumscribed around $E_0$ as follows. 
The initial polygon is a square.  In each iteration  we 
cut off a corner 
of the polygon with the largest~$P$-norm. So, we omit one vertex and 
add two new vertices. The cutting is by a line touching $E_0$ orthogonal 
to the segment connected to that vertex with the centre.

Let us denote by~$\nu_j$ the largest $P$-norm of vertices after the~$j^{th}$ iteration 
(the initial square corresponds to~$j=0$).  Since the norm is convex, its maximum on a polygon
is attained at one of its vertices. Hence, 
the norm of the cut  vertex is not less than the norm of each of the new vertices. 
Therefore, $\nu_{j+1} \le \nu_j$, so  the sequence 
 $\{\nu_j\}_{j\ge 0}$ is nonincreasing.  If at some step we have~$\nu_j \le 1$, 
 then all the vertices of the polygon after $j$ iterations are inside~$P$. 
 Hence, this polygon is contained in~$P$ and therefore~$E_0 \subset P$. 
 
 Otherwise, if $\nu_j > 1$, we have  $E_0 \not \subset   \nu_j \, \cos (\tau)\, 
 P$, 
 where $\tau$ is the smallest exterior angle of the resulting polygon.
 This is proved in Theorem~\ref{th.40} below. 
 Thus, the algorithm solves Problem~\ref{problem_ee} with the approximation factor~$q \ge 
 \nu_j \, \cos (\tau )$. 

{\em Comments}. In each iteration we need to find the vertex with the 
maximal $P$-norm. 
Therefore, we need to compute a norm of each vertex by solving Problem~\ref{problem_pe}. 
For this, we  compute the norms of two new vertices in each iteration.
Due to the central symmetry, one can reduce computation twice. 
Among two symmetric vertices we compute the norm of one of them and in each iteration we 
cut off both symmetric vertices.

Let $\tau$ be an arc of the unit circle connecting points~$\alpha$ and
$\beta$, we assume that $\tau < \pi$. Denote by~$\sigma = \sigma(\tau)$ the midpoint of~$\tau$
and 
$$
\bw(\tau) \  = \  
\frac{1}{\cos \bigl(\tau/2 \bigr)}\, 
\Bigl( \, 
\ba_0 \cos \sigma \, + \, \bb_0 \sin \sigma
\Bigr)\, . 
$$
Two lines touching~$E_0$ at the points corresponding to the ends of the 
arc~$\tau$ meet at~$\bw$.

\subsection*{The algorithm}

\subsubsection*{Initialization}

Choose  the maximum number of iterations~$J$. 
We split the upper unit semicircle (the part of the unit circle in the 
upper coordinate half-plane) into equal arcs~$\tau_1, \tau_2$ 
and compute the $P$-norms of the points~$\bw(\tau_i), \, i = 1,2$. 
Denote by~$\nu_0$ the maximum of those two norms and 
set $\cT = \{\tau_1, \tau_2\}$.

\subsubsection*{Main loop -- the $j^{th}$ iteration} 

We have a collection  $\cT$ of $j+1$ disjunct open arcs
forming the upper semicircle, the $P$-norms 
of all $j+1$ points~$\bw(\tau), \, \tau \in \cT$, and 
the maximal norm~$\nu_{j-1}$. 
Find an arc~$\tau$ with the biggest $P$-norm and 
replace that arc by two its halves~$\tau_1, \tau_2$. Update~$\cT$ and compute the $P$-norms of 
the points $\bw(\tau_1)$ and $\bw(\tau_1)$.  Set $\nu_j$
equal to the maximum of those two norms and of~$\nu_{j-1}$.   
\begin{itemize}
\item 
If $\nu_j \le 1$, then $E_0 \subset P$ and STOP. 

\item
If $\nu_j > \frac{1}{\cos (\tau)}$, where $\tau$
is the minimal arc in~$\cT$, then $E_0 \not \subset  P$ and STOP.

\item
If $1 <  \nu_j  \le \frac{1}{\cos (\tau)}$ and $j=J$, then $E_0 \not \subset \cos \,(\tau) \, P$. 

\item  
Otherwise go to the next iteration. 
\end{itemize}
  
\begin{thm}\label{th.40}
The corner cutting algorithm after $j$ iterations 
solves Problem~\ref{problem_ee} with the 
approximation factor~$q \ge \nu_j\cos(\tau)$, where 
$\tau$ is the minimal arc in~$\cT$.  
\end{thm}
\begin{proof}
Let~$\tau$ be the smallest arc  after $j$ iterations. Suppose this arc appears after the
$k^{th}$ iteration, $k\le j$. Then, its mother arc (let us call it~$2\tau$) had the biggest 
value $\|\bw(\vardot)\|_{P}$ among all arcs in the $k$th iteration. 
This means that $\|\bw(2\tau)\|_{P} = \nu_k$. Since the sequence~$\{\nu_i\}_{i\ge 0}$ is 
nonincreasing, we have $\nu_k \ge \nu_j$. 
The point~$\bx = \cos(\tau)\, \bw(2\tau)$ lies on~$E_0$. 
It does not belong to~$P$ precisely when~$\|\bx\|_P > 1$, i.e.\ when 
 $\bw(2\tau)  \, > \, \frac{1}{\cos (\tau)}$.  
 Thus, if $\nu_j  \, > \, \frac{1}{\cos (\tau)}$, then 
 $\bw(2\tau) \, = \, \nu_k   \, > \, \frac{1}{\cos (\tau)}$, and hence 
 $\bx \notin P$. Therefore, the inequality $\nu_j  \, > \, \frac{1}{\cos (\tau)}$
 implies that $E_0$ is not contained in~$P$.
\end{proof}

 The length of each arc has the form~$2^{-s}\pi$, where $s$ is the number of double 
 divisions to arrive at that arc. We call this number the {\em level} of the arc. 
 So, the original arcs of length $\nicefrac{\pi}{2}$  are those of level one. 
   
\subsubsection*{The complexity of the corner cutting algorithm}

To perform $j$ iterations 
one needs to solve Problem~\ref{problem_pe} for $j+2$ points~$\bw(\vardot)$. 
So, the complexity of the algorithm is defined by the complexity of solution
of Problem~\ref{problem_pe}. Below, 
in Sections~\ref{subsec_methodD-conic} and~\ref{sec_methodE} we derive two 
methods of its solution,
based on different ideas and compare them by numerical experiments. 
The approximation factor is $\cos (\tau) \, = \, \cos (2^{-s}\pi)\, = \, 
1 \, - \, 2^{-2s-1}\pi^2 \, + \, O(2^{-4s})$, where $s$ is the maximal level 
of the intervals after $j$ iterations.  Already for 
$s=2$ (after \emph{one} iteration) the approximation factor is~$q= \cos (\frac{\pi}{4}) =  \frac{\sqrt{2}}{2}$, 
which is better than in the complex polytope method, where  $q = \frac12$. 
For $s=3$ (after \emph{at most} three iterations),
 we have~$q = \cos (\frac{\pi}{8}) =  0.923\ldots$,
for $s=5$, we have~$q= 0.995\ldots$, for $s=10$, we have~$q > 1 - 10^{-5}$. 
In the worst case reaching the level~$s$ requires $j=2^s-1$ iterations. However, in practice 
it is much faster. Numerical experiments show that $j$ usually does not exceed $s+2$. 

In each iteration of the corner cutting algorithm we need to find the $P$-norm of the 
newly appeared vertices of the polygon. This means that we solve Problem~\ref{problem_pe}
for those vertices. 
The way of arriving at the solution actually defines the efficiency of the whole algorithm. 
We present two different methods and compare them.

\subsection{Solving Problem~\ref{problem_pe} via conic programming}
\label{subsec_methodD-conic}  

The norm $\|\bw\|_P$ is equal  to the minimal $r\in\re$
such that $\bw \in r P$, i.e.\
the minimal possible sum of nonnegative numbers 
$r_1, \ldots , r_j\in\re$ such that 
 $\bw = \sum_{j=1}^N r_j E_j$. 
 Thus, we obtain
 \begin{equation}\label{eq.w-conic-ell}
\left\{
\begin{aligned}
&r =  \min\, \sum_{j=1}^N r_j \quad\text{subject to}\\
&\tau_j \in [0, 2\pi), \quad j = 1,\ldots , N, \\
&\sum_{j=1}^N r_j\ba_j\cos \tau_j\,  \, + \, r_j\bb_j\sin \tau_j \ = \ \bw,\\
&\, r_j  \ge 0\, , \quad j=1, \ldots , N\, .\\
\end{aligned}
\right. 
\end{equation}
Changing variables $c_j = t_j\cos\tau_j, \, s_j = t_j\sin\tau_j $
 we obtain the conic programming problem 
\begin{equation}\label{eq.w-conic}
\left\{
\begin{aligned}
\text{minimize } &\sum_{j=1}^N r_j \quad\text{subject to}\\
&\sum_{j=1}^N c_j\ba_j + s_j\bb_j \ = \ \bw,\\
&\sqrt{c_j^2 + s_j^2} \ \le  \ r_j\, , \quad j=1, \ldots , N.\\
\end{aligned}
\right. 
\end{equation}
  with $3N$ variables $r_j, c_j, s_j\in\re$ and 
  $N(d+2)$ constraints. Among these constrains, there are 
  $N(d+1)$ linear and only $N$ quadratic ones, but the latter 
actually  defines the complexity of this problem. 
The problem is solved by conic programming. 
This can be done efficiently for dimensions $d\le 20$ and number of ellipsoids $N \le 1000$.

  The value $r  = \, \min\, \sum_{j=1}^N r_j $ of the problem~\eqref{eq.w-conic} 
  is equal to the norm~$\|\bw\|_P$. In particular, $\bw \in P$
  if and only if~$r\le 1$. 
  
  In the next section we introduce the second approach,
   when  the conic programming~\eqref{eq.w-conic} problem is approximated 
  with a linear programming one with precision 
  that increases exponentially with the number of extra variables.

\section{The projection method}  
\label{sec_methodE}

The corner cutting  method makes  use of a polygonal approximation 
of the ellipse~$E_0$. Can we go further and approximate all the $N$ ellipses 
$E_1, \ldots , E_N$ and thus approximate Problem~\ref{problem_pe} 
with a linear programming (LP) problem?
In principle, this is possible, but very inefficient.
Cutting  corners of $N$ polygons is expensive and slow.
If we do not involve cutting but just approximate 
each ellipse by a polygon,  the situation 
will be still worse due to a large total number of vertices of all polygons. Nevertheless, 
each approximating polygon can be build much cheaper if we present it as a projection 
 of a higher dimensional polyhedron. This technique was suggested by 
 Ben-Tal and  Nemirovski~\cite{BTN01} for approximating quadratic problems by LP problems. 
 See also~\cite{G10} for generalizations to other classes of functions.  
 We briefly describe this method (with slight modifications) and then 
 apply it to Problem~\ref{problem_pe}. Note that in contrast to the conic programming, here we 
 obtain only an approximate solution of Problem~\ref{problem_pe}.
 This is, however, not a restriction, since  
 the corner cutting algorithm also gives only an approximate solution 
 for Problem~\ref{problem_ee}. 
 If $q_1$ and $q_2$ are approximation factors of those two problems, then 
 the resulting approximation factor is~$q_1q_2$. If $q_i = 1-\varepsilon_i$
 with a small $\varepsilon_i, \, i=1,2$, then~$q_1q_2\, > \, 1-\varepsilon_1-\varepsilon_2$.

\subsection{A fast approximation of ellipses}  

The projection method realizes a polygonal approximation of ellipses 
by solving a certain LP problem
and  the precision of this approximation increases
exponentially in the LP problem input. 
This is done by an iterative algorithm, whose main loop is a doubling of a convex figure. 

\subsubsection*{Doubling of a figure}
 Consider an arbitrary  figure~$F\subset\re^2$ 
located in the lower half-space of the Cartesian plane. 
Then the set 
\begin{equation}\label{eq.double0}
F_0  \ = \bigl\{ (x', y')^T \ \bigr| \ x' = x, \, |y'| \le -y ,  \   (x, y)^T  \in F\bigr\}
\end{equation}
is the convex hull of~$F$ with its reflection about the abscissa, see Figure~\ref{fig_doubling}.
Indeed, each point~$A = (x, y)^T  \in F$ produces a vertical  segment
$\{(x , y') \ | \  y' \in [y, -y]\}$ which connects~$A$ with its reflection~$A'$
about the abscissa. Those segments fill the set~$F_0$. 

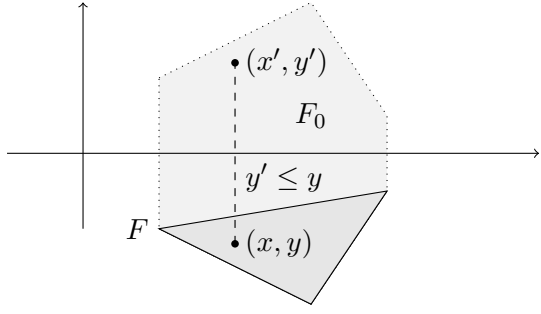
\begin{figure}
\small
    \centering
    \newcommand\Ax{-1}  \newcommand\Ay{-1}
    \newcommand\Bx{ 1}  \newcommand\By{-2}
    \newcommand\Cx{ 2}  \newcommand\Cy{-.5}
    \newcommand\ptx{ 0}  \newcommand\pty{-1.2}
    
    \begin{tikzpicture}

    \coordinate (abscm) at (-3,  0 );
    \coordinate (abscp) at (4,  0 );
    \coordinate (A)  at (\Ax, \Ay);
    \coordinate (B)  at (\Bx, \By);
    \coordinate (C)  at (\Cx, \Cy);

    \draw[-] (A) -- (B);
    \draw[-] (B) -- (C);
    \draw[-] (C) -- (A);

    \draw[draw=black, fill=gray!10, dotted] (\Ax,-\Ay) -- (\Bx,-\By) -- (\Cx,-\Cy) -- (C) -- (B) -- (A) -- cycle;
    \node at (1,.5) {$F_0$}; 
    \draw[draw=black, fill=gray!20 ] (A) -- (B) -- (C) -- cycle node[left] {$F$};
    
    \draw[->] (-3,0) -- (4,0);
    \draw[->] (-2,-1) -- (-2,2);
    \fill (\ptx,\pty) circle (.5mm) node[right] {$(x,y)$};
    \fill (\ptx,-\pty) circle (.5mm) node[right] {$(x',y')$};
    \draw[line width=0.15mm, dashed] (\ptx,\pty) -- (\ptx,-\pty) node[midway, below right]{$y'\leq y$};

    \end{tikzpicture}
    \caption{Set $F$ and the convex hull with its reflection at the abscissa $F_0={\rm co}\{F,F'\}$.}%
    \label{fig_doubling}
\end{figure}

In the same way one can double a figure~$F$ about an arbitrary line
passing through the origin provided~$F$ lies on one side with respect to this line. 
Let a line $\ell_{\alpha}$ be defined by the equation~$y = x\tan \alpha$; 
it makes the  angle~$\alpha \in\bigl[ 0, \pi \bigr]$ with the abscissa. 
After the clockwise rotation by the angle~$\alpha$ the line $\ell_{\alpha}$ 
becomes the abscissa and $F$ becomes a figure~$F'$ located in the lower half-plane. 
Since this rotation is defined by the matrix
$$
R_{\alpha} \ = \ 
\left(
\begin{array}{rr}
\cos \alpha & \sin \alpha\\
-\sin \alpha & \cos \alpha
\end{array}
\right)\, , 
$$
it follows from formula~\eqref{eq.double0} that 
the figure~$F_{\alpha}$, the convex hull of~$F$ with its reflection about the 
line~$\ell_{\alpha}$,
consists of points~$(x_1, y_1)$ 
satisfying the following system  of inequalities: 
\begin{equation}\label{eq.double-alpha}
\left\{
\begin{aligned}
x_1 \cos \alpha \, + \, y_1 \sin \alpha & = x \cos \alpha \, + \, y \sin \alpha \\
\bigl| - x_1 \sin \alpha \, + \, y_1 \cos \alpha \bigr| & \le
 x \sin \alpha \, -  \, y \cos \alpha\\
(x, y)^T \in F
\end{aligned}
\right. 
\end{equation}

\subsubsection*{Construction of a regular $2^n$-gon} 
Now we describe the algorithm of recursive doubling of a polygon. 

We take an arbitrary radius~$r > 0$, denote $\alpha_m = 2^{-m}\, \pi \, , m\ge 0$, and 
consider  an isosceles  triangle~$AOB$, 
where $A=(r,0)^T$, 
$B = (r\cos \alpha_n , r\sin \alpha_n)^T$, 
and $O$ is the origin. 
Double this triangle about the line~$\ell_{\alpha_n} = OB$, then 
double the obtained quadrilateral about~$\ell_{\alpha_{n-1}}$
(the lateral side different from~$OA$), then about~$\ell_{\alpha_{n-2}}$, etc..
After $n$ doublings (the last one is about $\ell_1$, which is abscissa)
we get the regular $2^n$-gon inscribed in the circle of radius~$r$. 
We denote this polygon by~$rT_n$. Thus, $T_n$ is the 
$2^n$-gon inscribed in the unit circle. 
Note that the initial triangle~$AOB$ is defined by the 
system of linear inequalities
~$0\le y\le x\tan \alpha_n$ and 
$x+y\tan \alpha_{n+1} \le r$. 

Thus, we obtain the following description of the set~$rT_n$, which is a regular 
$2^n$-gon inscribed in the circle of radius~$r$:  
$$
rT_n \ = \ \bigl(x_{2n+1}, x_{2n+2} \bigr)^T \ : 
$$
\begin{equation}\label{eq.rTn}
\left\{
\begin{aligned}
&0 \ \le \ x_2 \ \le \ x_1\tan \alpha_{n-1}&&\\
& x_1\ +\ x_2 \tan \alpha_{n}\  \le \ r&&\\
&\mbox{for } \ k \ = \ 1, \dots , n\, :&&\\
&\phantom{\big|+\,}x_{2k+1} \cos \alpha_{n-k} \, + \, x_{2k+2} \sin \alpha_{n-k} \phantom{\big|}& =
x_{2k-1} \cos \alpha_{n-k} \, + \, x_{2k} \sin \alpha_{n-k} \\
&\big| - x_{2k+1} \sin \alpha_{n-k} \, + \, x_{2k+2} \cos \alpha_{n-k} \big| & \le
 x_{2k-1} \sin \alpha_{n-k} \, -  \, x_{2k} \cos \alpha_{n-k} \\
\end{aligned}
\right. 
\end{equation}
This is a linear system of inequalities with variables~$r, x_1, \ldots , x_{2n+2}$. 
The inequality with modulus $|a|\le b$ is replaced by the system 
$a \le b, \, -a \le b$.  The system~\eqref{eq.rTn}
consists of $3n+3$ linear constraints (equations and inequalities) with  $2n+3$ variables. 
For all vectors $X = (x_1, \ldots , x_{2n+2})^T$ satisfying system~\eqref{eq.rTn}, 
the vector composed by the two last components $(x_{2n+1}, x_{2n+2})^T$
fills the regular $2^n$-gon. So, this  $2^n$-gon is a projection of 
a $(2n+2)$-dimensional polyhedron to the plane. This 
polyhedron has $3n+3$ facets.  

\subsubsection*{Construction of an affine-regular $2^n$-gon inscribed in an ellipse} 

For and arbitrary ellipse~$E(\ba, \bb)$, the 
point~$x_{2n+1}\ba \, + \, x_{2n+2}\bb$ runs 
over an affine-regular $2^n$-gon inscribed in the ellipse~$rE(\ba, \bb)$
as the vector~$X = \bigl(x_1, \ldots , x_{2n+1}, x_{2n+2}\bigr)$ 
runs over the set of solutions of the linear system~\eqref{eq.rTn}
 with this value of~$r$.

\subsection{Solving Problem~\ref{problem_pe} by the fast polygonal approximation}  

We approximate all ellipses $E_j = E(\ba_j, \bb_j), \, j = 1, \ldots , N$
by polygons and then decide if~$\bw \in P$ with some approximation factor.

We fix a natural~$n$ and nonnegative numbers~$r^{(1)}, \ldots , r^{(N)}$ such that 
$\sum_{j=1}^N r^{(j)} = 1$. For each $j$, we consider  the affine-regular polytope 
$$
r_jT_{n}^{(j)} \ = \ x_{2n+1}^{(j)}\ba_j \, + \, x_{2n}^{(j)}\bb_j
$$
inscribed in~$E_j$, 
where 
$$
\Bigl(r_j , X^{(j)}\Bigr) \ = \ \Bigl(r_{j}, x_1^{(j)}, \ldots , x_{2n+1}, x_{2n+2}^{(j)}\Bigr)^T
$$ 
is a feasible vector for the linear system~\eqref{eq.rTn}. 
If~$\bw \in r^{(1)}T_n^{(1)} + \cdots + r^{(N)}T_{n}^{(N)}$, then 
$\bw \in r^{(1)}E_1 + \cdots + r^{(N)}E_N$. Therefore, 
$\bw \in P$ whenever there exist 
 numbers $r^{(j)}\ge 0$ such that $\sum_{j=1}^N r^{(j)} = 1$ and  
 $\bw \in r^{(1)}T_n^{(1)} + \cdots + r^{(N)}T_{n}^{(N)}$. 
 Hence,  the assertion $\bw \in P$ is decided by the following LP problem: 
  \begin{equation}\label{eq.w-lin}
\left\{
\begin{array}{lcll}
\displaystyle\sum_{j=1}^N r^{(j)} \quad \to \quad \min
{} & {} & {} & {}\\
0 \ \le \ x_2^{(j)} \ \le \ x_1^{(j)}\tan \alpha_{n-1}, & {} & {}\\
{} & {} & {} & {}\\
 x_1^{(j)}\ +\ x_2^{(j)} \tan \alpha_{n}\  \le \ r^{(j)} , & {} & {}\\
 {} & {} & {} & {}\\
\phantom{\bigl|-}x_{2k+1}^{(j)} \cos \alpha_{n-k} \, + \, x_{2k+2}^{(j)} \sin \alpha_{n-k}\phantom{\bigl|} & = &  
x_{2k-1}^{(j)} \cos \alpha_{n-k} \, + \, x_{2k}^{(j)} \sin \alpha_{n-k}, \\
{} & {} & {} & {}\\
\bigl| - x_{2k+1}^{(j)} \sin \alpha_{n-k}^{(j)} \, + 
        \, x_{2k+2}^{(j)} \cos \alpha_{n-k} \bigr| & \le &   
 x_{2k-1}^{(j)} \sin \alpha_{n-k} \, -  \, x_{2k} \cos \alpha_{n-k}^{(j)},\\
   {} & {} & {} & {}\\
 r^{(j)}\ \ge \  0\ ; & {} & {} &{}\\
 {} & {} & {} & {}\\
 \ k \ = \ 1, \dots , n, 
  \ j \ = \ 1, \dots , N,& {} & {} & {}\\
   {} & {} & {} & {}\\
\displaystyle \bw \ = \ \sum_{j=1}^N x_{2n+1}^{(j)}\ba_j \, + \, x_{2n+2}^{(j)}\bb_j\ ,& {} & {} & {}
\end{array}
\right. 
\end{equation}
in the variables $r^{(j)}, x_{s}^{(j)}$, $j=1, \ldots , N$, $s=1, \ldots , 2n+2\, . 
 $   Let us remember that~$\alpha_m = 2^{-m}\pi$. 
 The value of this problem~$r \, = \, \sum_{j=1}^N r^{(j)}$ 
 is the minimal number such that $\bw$ belongs to the 
 set $rP_n$, where $P_n = {\rm co}\, \{T_n^{(1)}, \ldots , T_n^{(N)}\}$. 
 In other words, $r = \|\bw\|_{P_n}$. 
 In particular, $\bw \in P_n$ precisely when $r\le 1$.

The LP problem~\eqref{eq.w-lin} 
has  $(2n+3)N$ variables 
$r^{(j)}, \, x_{s}^{(j)}$ and 
$(3n+4)N + d+1$ linear  constraints (equations and inequalities). 
Note that the matrix of this problem possesses only~$(12n + 2d + 7)N  +d $ nonzero coefficients, 
i.e.\ the total number of nonzero coefficients is linear in the size of the matrix. 
On the other hand,  the product of the number of variables times the number of constraints
exceeds $6n^2N^2 + 2Nd$. Thus, this problem  is very sparse. 

Since $P_n \subset P$, it follows that $\bw \in P$, whenever~$r\le 1$. 
In fact, problem~\eqref{eq.w-lin} provides an approximate solution 
to Problem~\ref{problem_pe} with the factor $q = \cos(2^{-n}\pi)$.
\begin{thm}\label{th.100}
If $r$ is the value of the LP problem~\eqref{eq.w-lin}, then
for every~$\omega \in \re^d$, we have 
$ \  r\, \cos (2^{-n}\pi) \, \le \, \|\bw\|_P \, \le \, r $. 
\end{thm}
\begin{proof}
Since the ratio between the radii of the inscribed and the circumscribed circles
of a regular $2^n$-gon is equal to~$q = \cos (2^{-n}\pi)$, we see that 
$E_j \subset q T^{(j)}_n$ for each~$j$. Consequently, $P \subset qP_n$
and hence $\|\bw\|_{P} \ge   \|\bw\|_{P_n}$, from which the theorem follows. 
\end{proof}

\begin{cor}\label{c.10}
If $r \le 1$, then  
$\bw \in P$, otherwise $\bw \notin \cos(2^{-n}\pi) P$. 
\end{cor}

Since  $\cos (2^{-n}\pi) \, = \, 1 \, - \, 2^{-2n-1}\pi^2\, + \, O(2^{-4n})$,
we see that already for small values of~$n$ we obtain a very sharp estimate.
The rate of approximation for $n \le 8$ is given in Table~1.  

\begin{table}
\begin{center}
\begin{tabular}{c|c c c c c c}
$n$ & 3 & 4 & 5 & 6 & 7  & 8 \\
\hline
$\cos (2^{-n}\pi)$ & 0.9238 & 0.9807 &  0.9951  & 0.9987 & 0.9996  & 0.9999  \\
\end{tabular}
\caption{ \small The partial approximation factor $q_1$ 
for Problem~\ref{problem_pe} for small~$n$ rounded to four decimal places}
\end{center}
\end{table}
For $n=12$,  we have~$q > 1  -   10^{-6}$;  for $m=17$,  we have~$q > 1 - 10^{-9}$.

\section{Numerical results}  
\label{sec_numerics}

\begin{figure}[p]
    \vspace*{0cm}
    \small
	\centering
    Complex polytope method\\
	\includegraphics[width=0.49\linewidth]{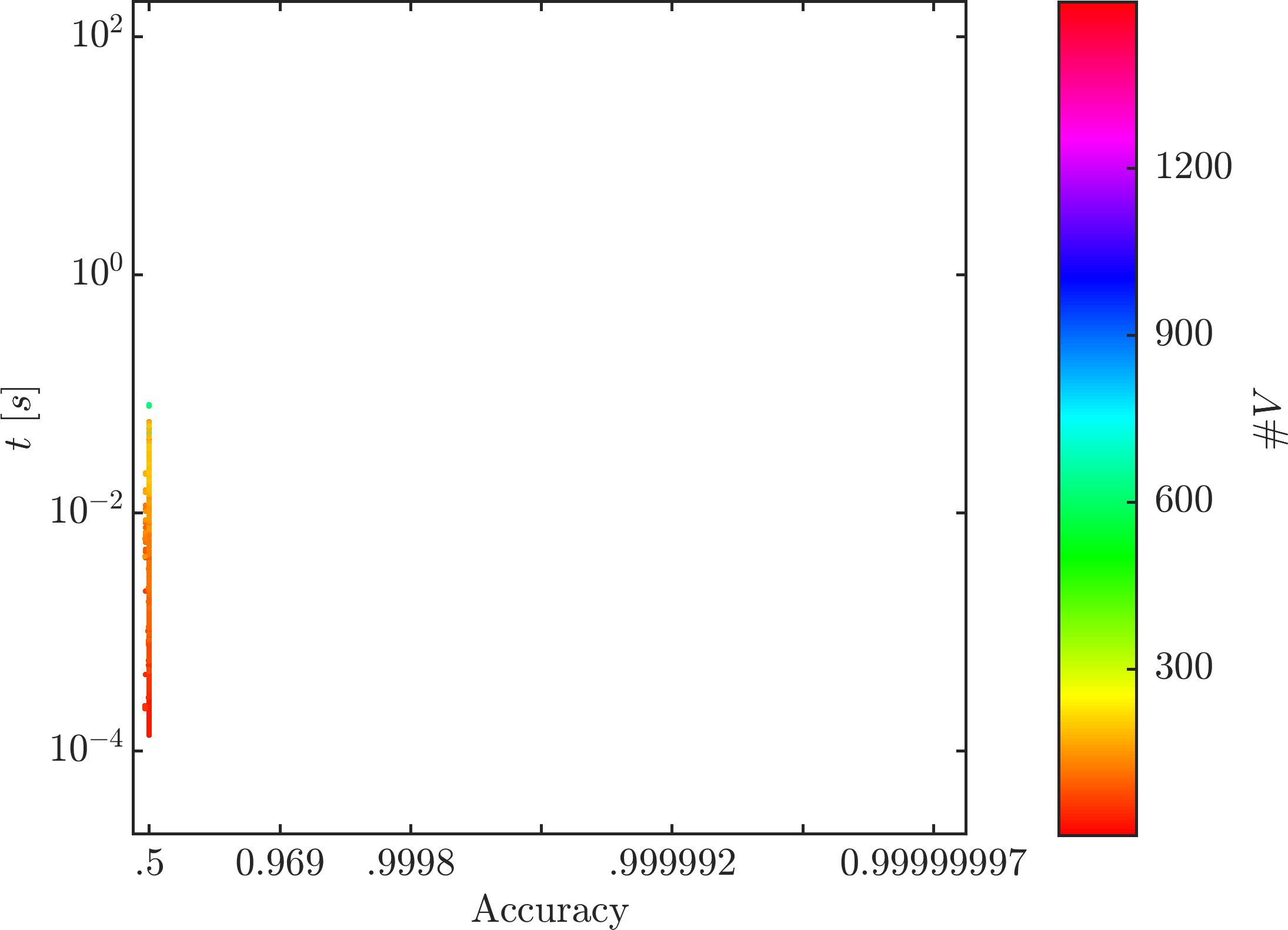}
   	\includegraphics[width=0.49\linewidth]{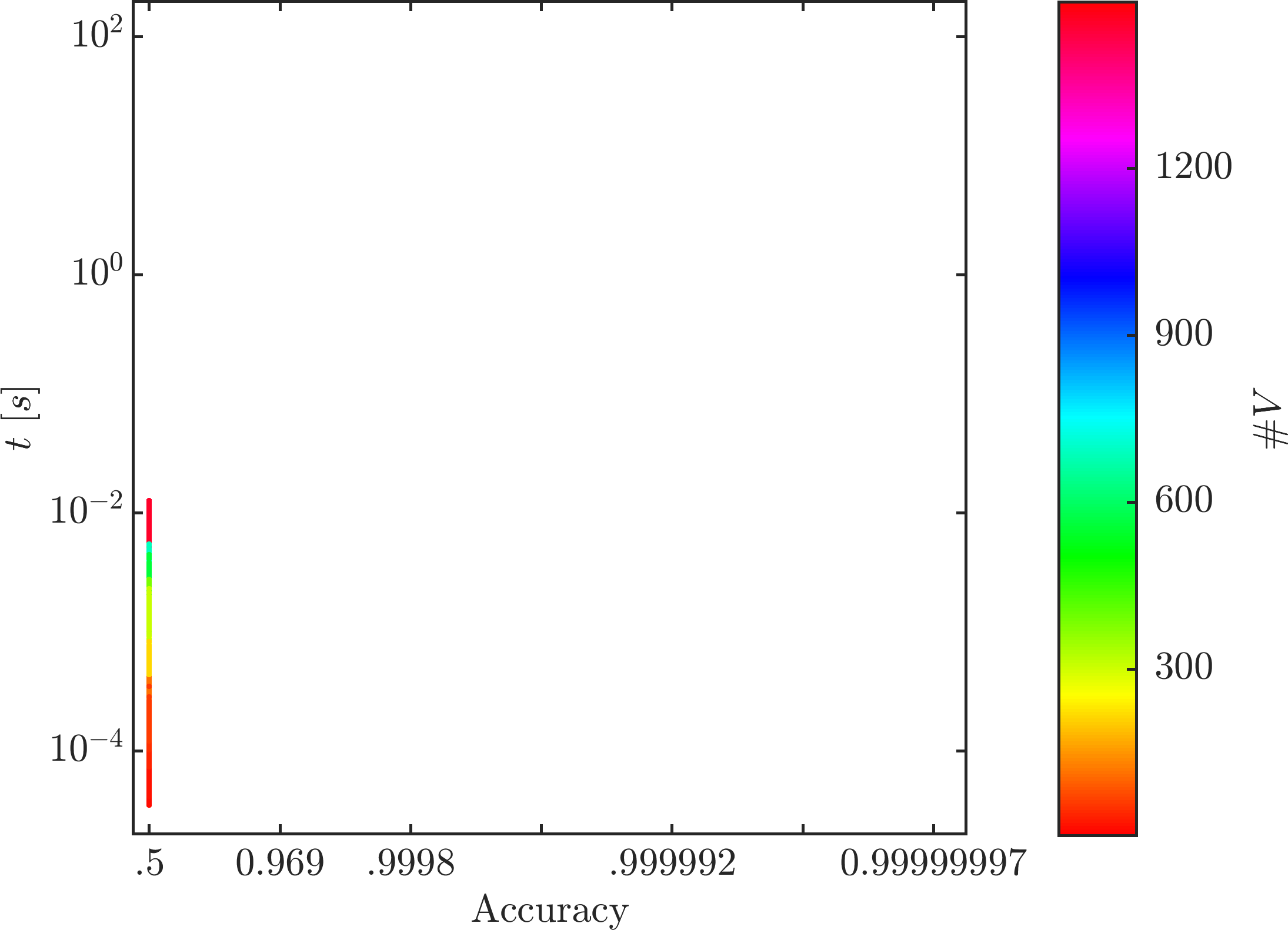}\\[2ex]
    Corner cutting method\\%
	\includegraphics[width=0.49\linewidth]{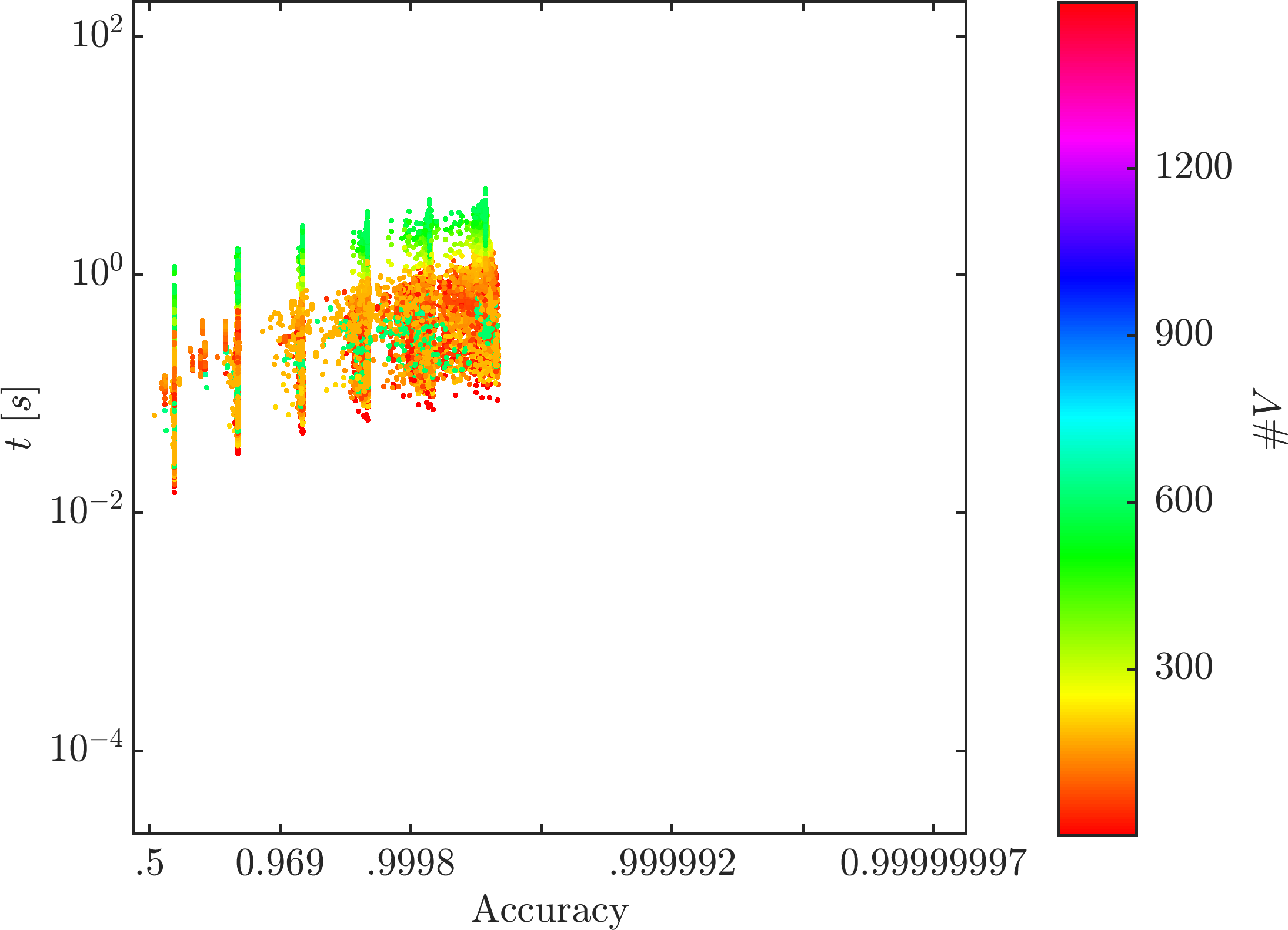}
   	\includegraphics[width=0.49\linewidth]{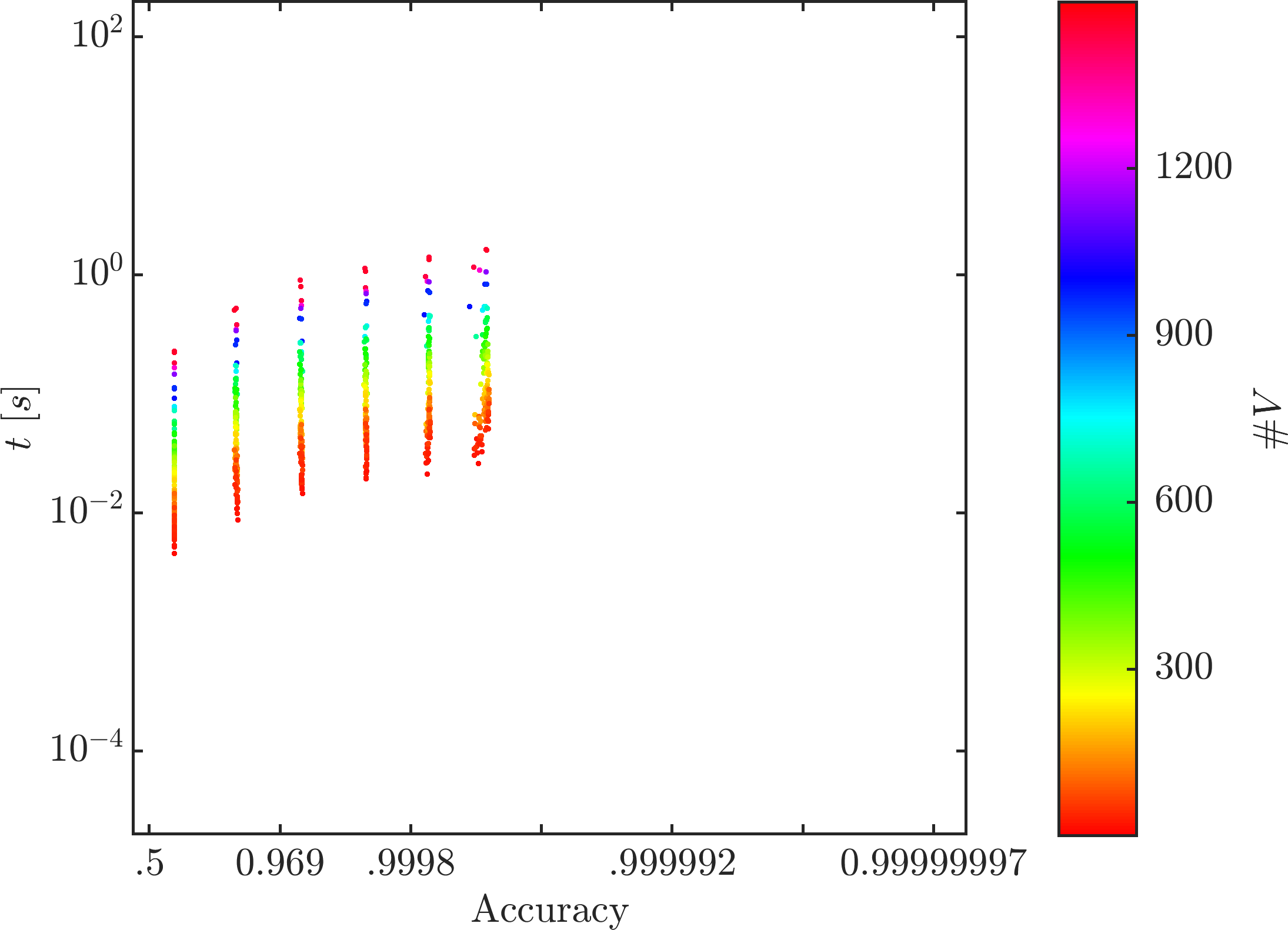}\\[2ex]
    Projection method\\
	\includegraphics[width=0.49\linewidth]{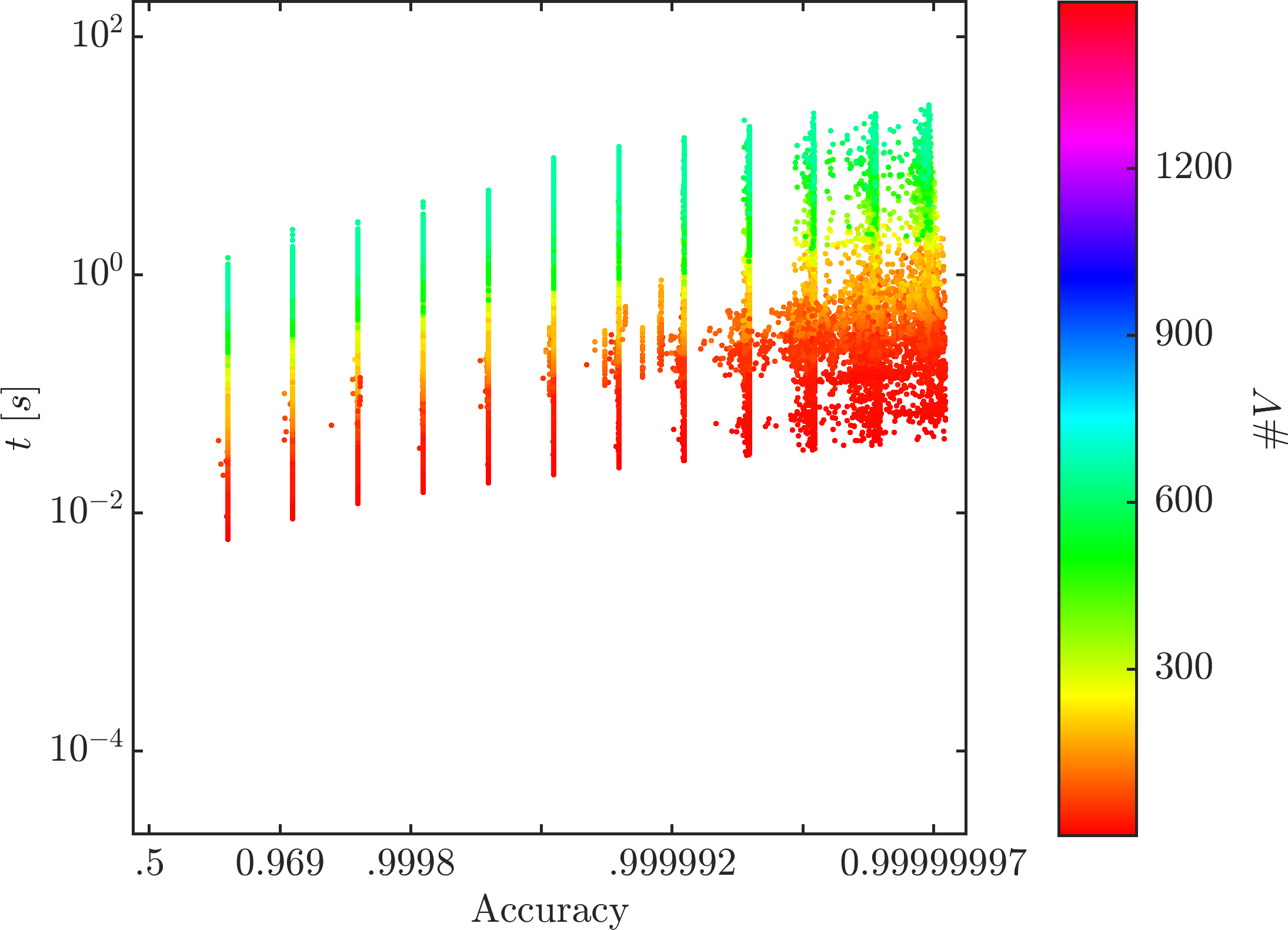}
   	\includegraphics[width=0.49\linewidth]{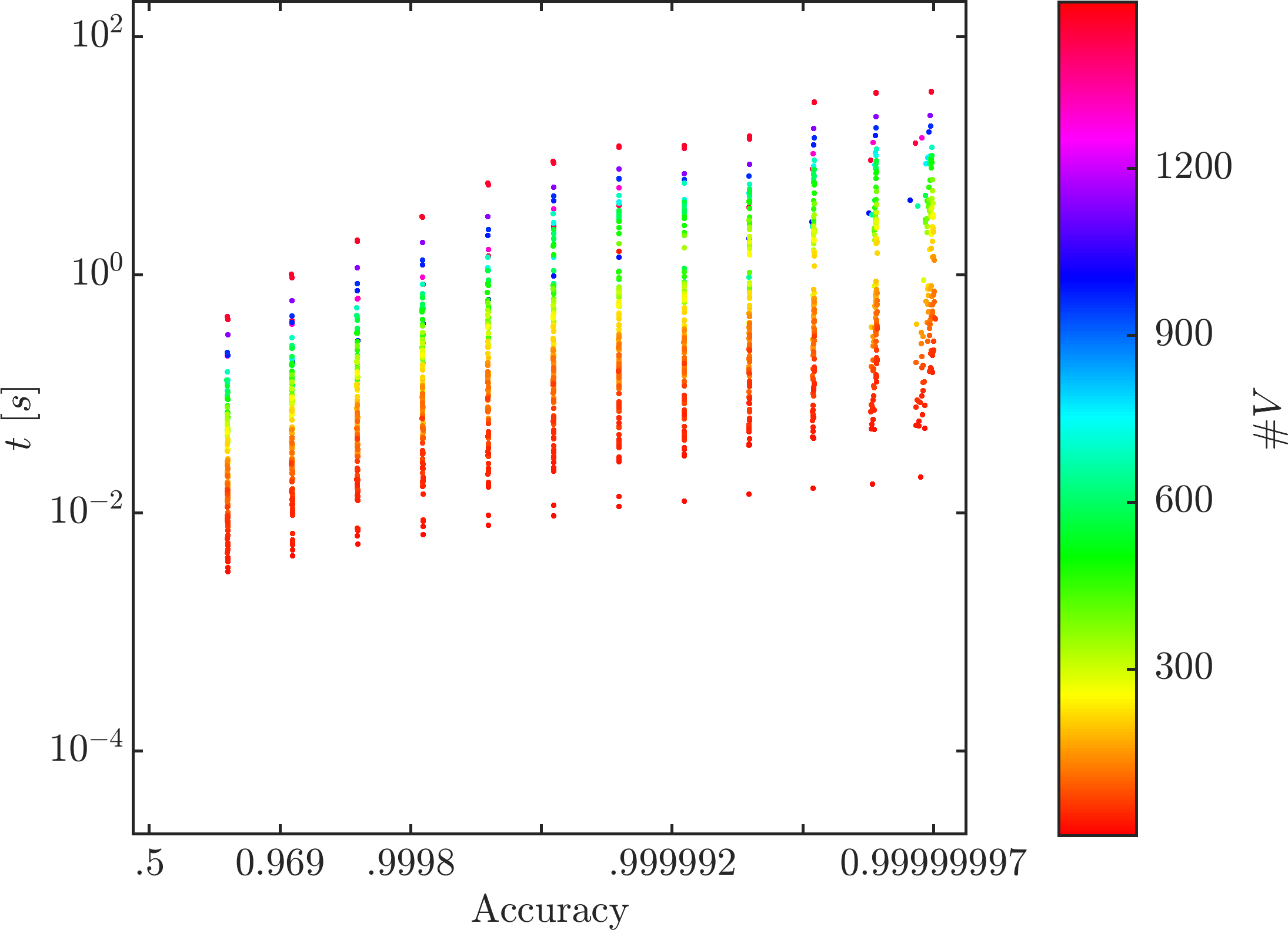}\\[2ex]~
    
    {
    \small 
     Figure~\ref{fig_comparison_caption}: Runtime $t$ in seconds of the the methods 
     \emph{complex polytope}, \emph{corner cutting} and \emph{projection}. 
     See the full caption at page~\pageref{fig_comparison_caption}.
     }

	\label{fig_comparison}
\end{figure}

\begin{figure*}

\small
    On the $x$-axis the theoretical minimal accuracy on a logarithmic scale is printed, 
    on the $y$-axis the time the algorithm needed, also on a logarithmic scale.
    The true, obtained, accuracy is especially for the complex polytope method much higher.
    The colour indicates the number of vertices of the elliptic polytope. 
    The dimension of the problem is not plotted,
     since it turned out to have only a very minor influence on the runtime.
     
    All algorithms were assessed using the same data set,
    the corner cutting method and the projection method were tested
     with different approximation factors.
     
    The left column is for data arising in the Invariant polytope algorithm.
    The right column is for data of ellipses and elliptic polytopes 
    with normal distributed real and imaginary part.

    One can see clearly, that the complex polytope method is the most efficient algorithm 
    when one compares the time the algorithm needs with its accuracy. 
    This is even more true under the viewpoint that the complex polytope method
    on average yields an accuracy of $0.7071$.
    
    Comparing the corner cutting method and the projection method, 
    one sees that the latter clearly outperforms the former consistently.
    
    Note: The blurring of the last accuracy values in each plot is due to numerical errors.

    \bigskip
\makeatletter
\small
\leftskip\z@\@plus 100fil%
\rightskip\z@\@plus -100fil%
\parfillskip\z@\@plus 200fil\relax    
    Caption for Figure~\ref{fig_comparison_caption} on page~\pageref{fig_comparison}:
    Runtime $t$ in seconds of the the methods 
    \emph{complex polytope}, \emph{corner cutting} and \emph{projection}.
    \par
\makeatother

\label{fig_comparison_caption}
\end{figure*}

In this section we demonstrate practical implementations 
of our methods of finding the convex hulls of ellipses.  
We use the following  solvers:  \emph{Matlabs} \texttt{linprog} and 
\emph{Gurobi}%
\footnote{\emph{Gurobi} is a commercial solver, but a free academic licence can be obtained at~%
\href{https://www.gurobi.com/academia/academic-program-and-licenses/}{\emph{gurobi.com}}.}
for the linear programming (\emph{LP}) problems and 
\emph{SeDuMi}%
\footnote{\emph{SeDuMi} is free and can be downloaded at~\href{https://github.com/sqlp/sedumi}{\emph{github.com/sqlp/SeDuMi}}. 
The GitHub version is a maintained fork of the original project,
whereas the original host does not seem to maintain SeDuMi any more.}
and \emph{Gurobi} for the 
quadratic programming (\emph{QP}) problems.  

We obtain numerical results and compare them for the following methods presented in this paper:  
\begin{itemize}
\item Complex polytope method (Section~\ref{sec_methodB})
\item Corner cutting method (Section~\ref{sec_methodD})
\item Projection method (Section~\ref{sec_methodE})
\item Mixed method
\end{itemize}
The \emph{Mixed method} is a combination of the complex polytope method 
and the projection method. 
The former is the fastest algorithms of all three, 
the latter is the most accurate. 
The mixed method accepts an additional parameter \emph{bound} 
describing the range of values one is interested in. 
Whenever the complex polytope method determined that the norm is inside or outside of the range of interest,
the exact algorithm is not started, and thus the computation is sped up.
For example, 
for the application of computing the joint spectral radius using the Invariant polytope algorithm,
one is only interested whether an ellipse lies inside or outside of the convex hull of the elliptic polytope. 
Now, whenever the complex polytope method concludes that an ellipse lies inside or outside, 
one can already  abort the computation.

The algorithms are implemented in Matlab and included in the \emph{ttoolbox}~\cite{ttoolbox}.
The scripts to generate and evaluate the data can be downloaded from~%
\href{https://tommsch.com/science.php}{\emph{tommsch.com/science.php}}
All software is thoroughly tested using the \emph{TTEST} framework~\cite{TTEST}.
The various implemented methods are optimized to a different degree,
and thus, timings cannot be compared well.

To obtain quantitative measures of how the methods differ, 
we generated two test sets of random ellipses and elliptic polytopes.

\emph{(Dataset~$A$)} The first set contains ellipses and elliptic polytopes
whose ellipses have normal distributed real and imaginary part.
Dataset~$(A)$ consists of 365 elliptic polytopes in dimension 3 to 25
and the norm is computed approximately for 12 ellipses per elliptic polytope.

\emph{(Dataset~$B$)} The second set is generated by the Invariant polytope algorithm, 
where we stored the intermediate occurring ellipses and elliptic polytopes
for some random sets of input matrices with complex leading eigenvalue. 
Dataset~$(B)$ consists of 119 elliptic polytopes in dimensions 2 to 12
and the norm is computed of 100 ellipses per elliptic polytopes.

For the tests we used a PC with an 
AMD Ryzen 3600, 6 cores%
\footnote{For the tests only 5 cores were used.},
3.6\,GHz, 64\,GB RAM, Windows 10 build 1809%
\footnote{Windows 10 build 1809 has problems with the used Ryzen 3600 processor. 
Newer versions of Windows run usually 10\% faster on this processor.}, 
Matlab R2020a, 
Gurobi solver 9.0.2 from May 2019,
SeDuMi solver 1.32 from July 2013,
ttoolboxes v1.2 from June 2021,
TTEST v0.9 from June 2021.

The measured runtime of the algorithms with respect to the chosen accuracy and number of vertices
can be seen in Figure~\ref{fig_comparison}.

\subsection{Behaviour of the complex polytope method}

\begin{figure}[!ht]
	\centering
	\includegraphics{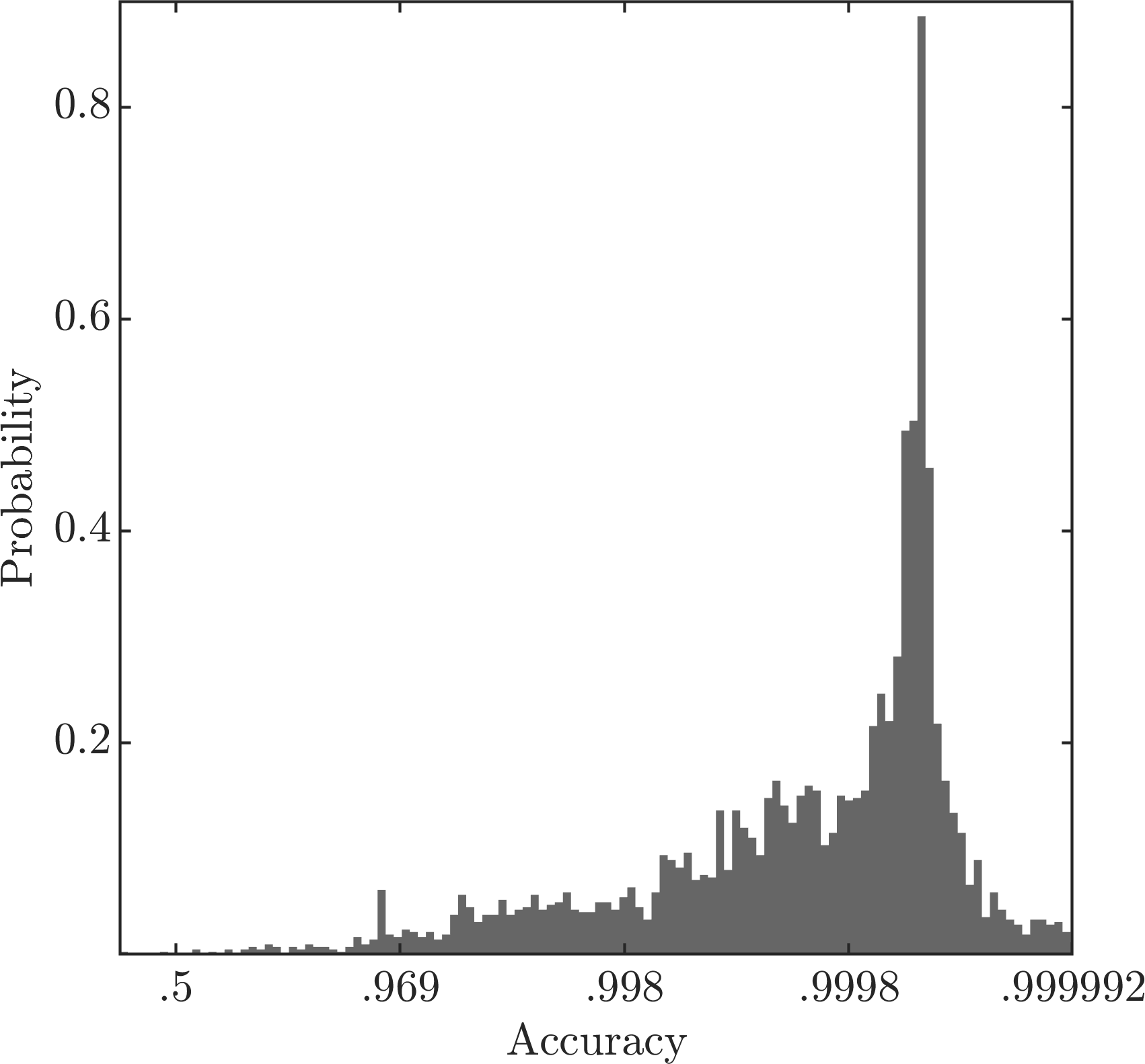}
   	\includegraphics{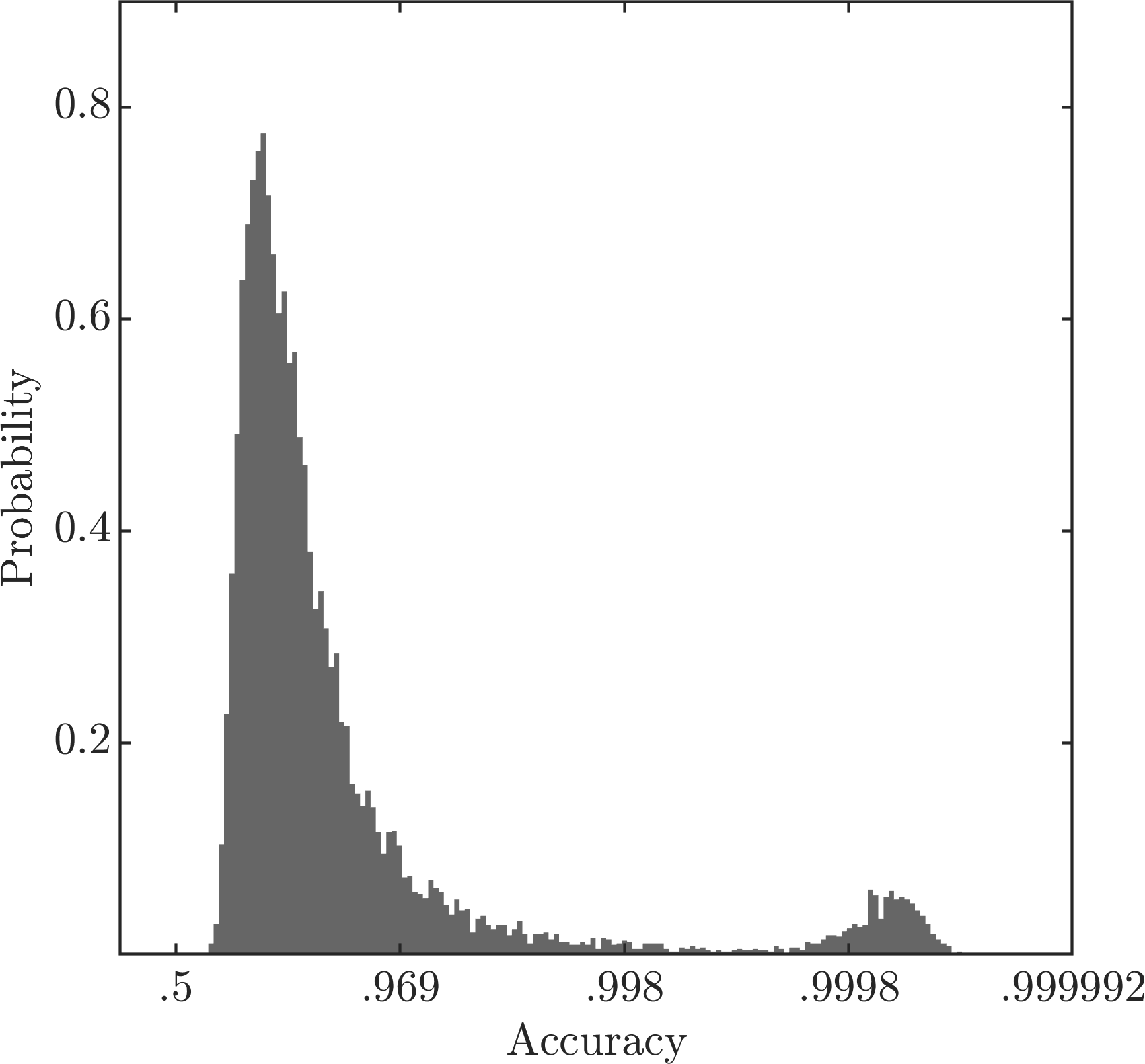}
	\caption{\small 
    Estimated probability density function of the approximation factor 
    of the complex polytope method for two different data sets.
    The left pictures data set is Dataset $(A)$,
    the right pictures data set is Dataset $(B)$.
    }
	\label{fig_accuracy_cb}
\end{figure}

Although the theoretical approximation factor of this method is $\nicefrac{1}{2}$,
in numerical experiments it turns out that the average approximation factor 
is mostly larger than $\nicefrac{1}{\sqrt{2}}$.
In small dimensions, $d=2,3$, the approximation factor is even close to $1$ in a lot of cases, 
see Figure~\ref{fig_accuracy_cb} for the (estimated) probability density function 
of this methods approximation factors.

Note that the numerical accuracy of the QP solver is approximately $10^{-5}$ and thus, 
the maximal accuracy which can be reached is approximately $0.99999$, 
which is quite exactly the position of the rightmost peaks
in Figure~\ref{fig_accuracy_cb}

\subsection{Behaviour of the corner cutting method}

The corner cutting method is, like the complex polytope method, a QP problem, and thus,
the absolute error of the solution returned by our numerical solvers is in the range of $10^{-5}$. 
For the corner cutting method, this accuracy is on average obtained after 10 to 12 iterations 
in the generic case, as our experiments show.
Apart from the chosen accuracy, the runtime of the algorithm mostly depends on, 
firstly, the geometry of the problem and, 
secondly, on the number of vertices of the elliptic polytope. 
The dimension of the problem only has a minor influence on the runtime.

\subsection{Behaviour of the Projection Method  (Method E)}

The absolute error of the LP solvers is roughly $10^{-9}$, 
which is magnitudes higher than for the QP solver.
Solely due to this fact, the projection  method 
is the most accurate method of all the described methods.

For the projection  method one could increase the number of vertices of the polytopes 
approximating the ellipses of the elliptic polytope until the norm is computed up to the desired accuracy, 
similar as in the corner cutting method.
Unfortunately, this hinders the use of warm-starting the LP problem since this alters the underlying LP.
Therefore, in our implementation we choose the approximation factor 
$q_1$ corresponding to Problem EE* to be of the same magnitude than the approximation factor 
$q_2$ corresponding to Problem EE, 
and such that $q_1q_2\simeq q$, where $q$ is the chosen accuracy.

\section{Applications}  
\label{sec_applications}

\subsection{Number of extremal vertices}

Before we demonstrate the main applications,
the construction of 
Lyapunov functions of linear systems and evaluation of extremal norms, 
we address the question of the expected number of vertices in the convex hull of random ellipses. 
This issue is important  for both of the above applications since 
it shows the growth of the number of ellipses with respect to the number of the iterations of the algorithms. 
  
The corresponding  problem for the convex hull of random points 
originated with the famous question of Sylvester~\cite{Sylv1864}.
The answer highly depends on the domain on which the points are sampled and on the dimension. 
Various lower and upper bounds on the asymptotically 
expected number of points in the convex hull are known, see~\cite{Ha99}\cite{DT09} and references therein. 
It would be extremely interesting to come up with similar theoretical
estimates for convex hulls of ellipses. Here we compare the two cases solely numerically.
There is no canonical analogue between points sampled from some domain and ellipses 
sampled from some domain, since the ellipses are determined by two parameters instead of one.
We introduce several numerical results with various samplings. 

\subsubsection*{Uniform sampled ellipsoids in the unit ball}

\begin{figure}[!ht]
	\centering
	\includegraphics{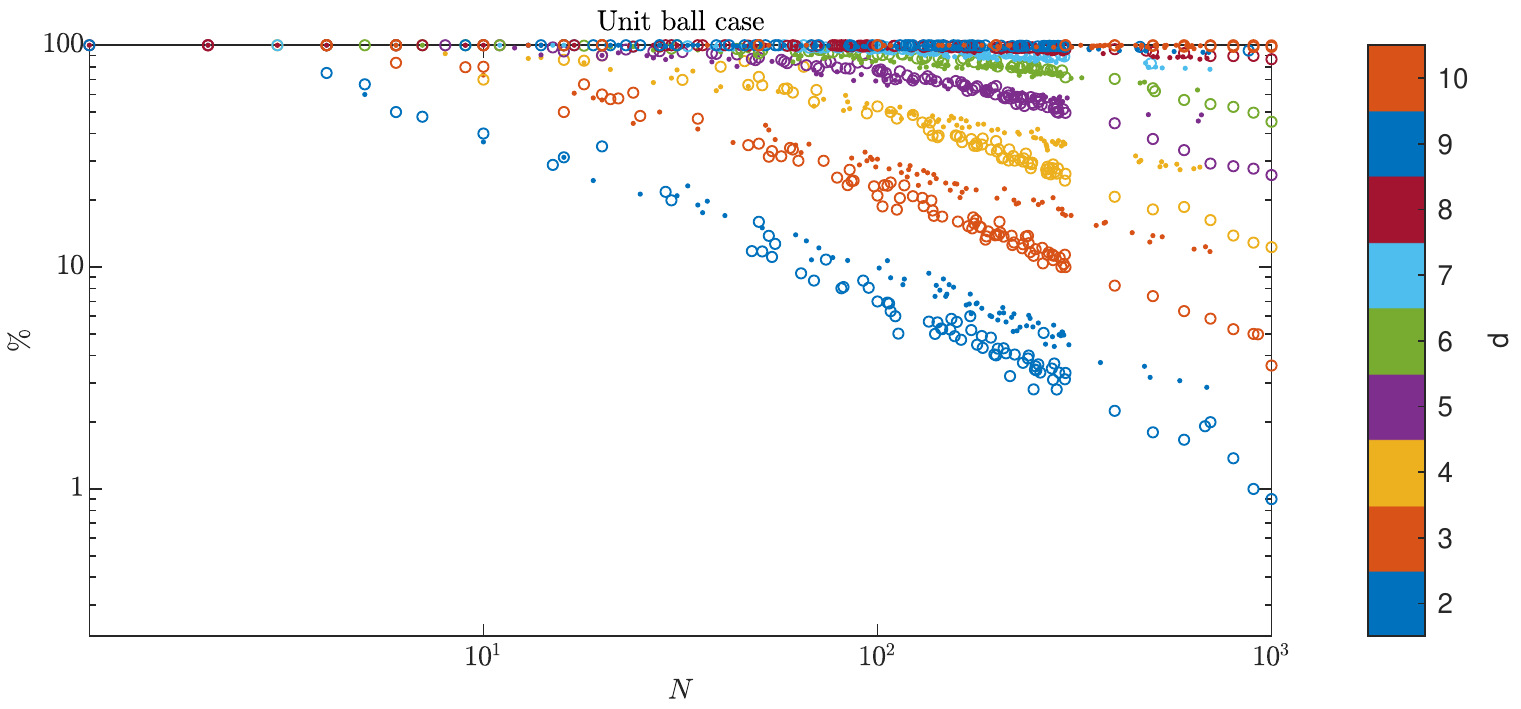}
	\caption{\small 
    Fraction of points or ellipses which belong to the convex hull 
    of randomly selected points or ellipses which are 
    uniformly distributed in the unit ball or have uniformly distributed 
    real and imaginary part in the unit ball, respectively.}
	\label{fig_extremal_ball}
\end{figure}

Given ellipses whose real and imaginary part are sampled uniformly from the unit ball,
and given points uniformly sampled from the unit ball. The number of vertices and ellipses 
of their corresponding convex hull is plotted in Figure~\ref{fig_extremal_ball}.
Experiments are made for dimensions 2 to 10 and number of points and ellipses 1 to 1000. 
Since the computational time increases significantly with the number of points or ellipses,
for sets with more than 300 points or ellipses less examples were conducted.
In the plot one can see the relative fraction of points or ellipses belonging to the convex hull, 
coloured with respect to the dimension.
The point examples are plotted with a $\cdot$ symbol,
the ellipse examples are plotted with a $\circ$ symbol.

Interestingly, the two cases differ greatly. 
Whereas for dimensions 2 to 5 the fraction of ellipses belonging to the convex hull 
is less than for the point counterpart,
the situation is reversed from dimension 7 upwards.

\subsubsection*{Uniform sampled ellipsoids in the unit cube}

Interestingly, when the points or the real and imaginary parts of the ellipses 
are sampled uniformly from a unit-cube,
the behaviour between the point case and the ellipses is very similar,
 at least for small dimensions, as can be seen in Figure~\ref{fig_extremal_cube}.
 
 \begin{figure}[!ht]
 	\centering
 	\includegraphics{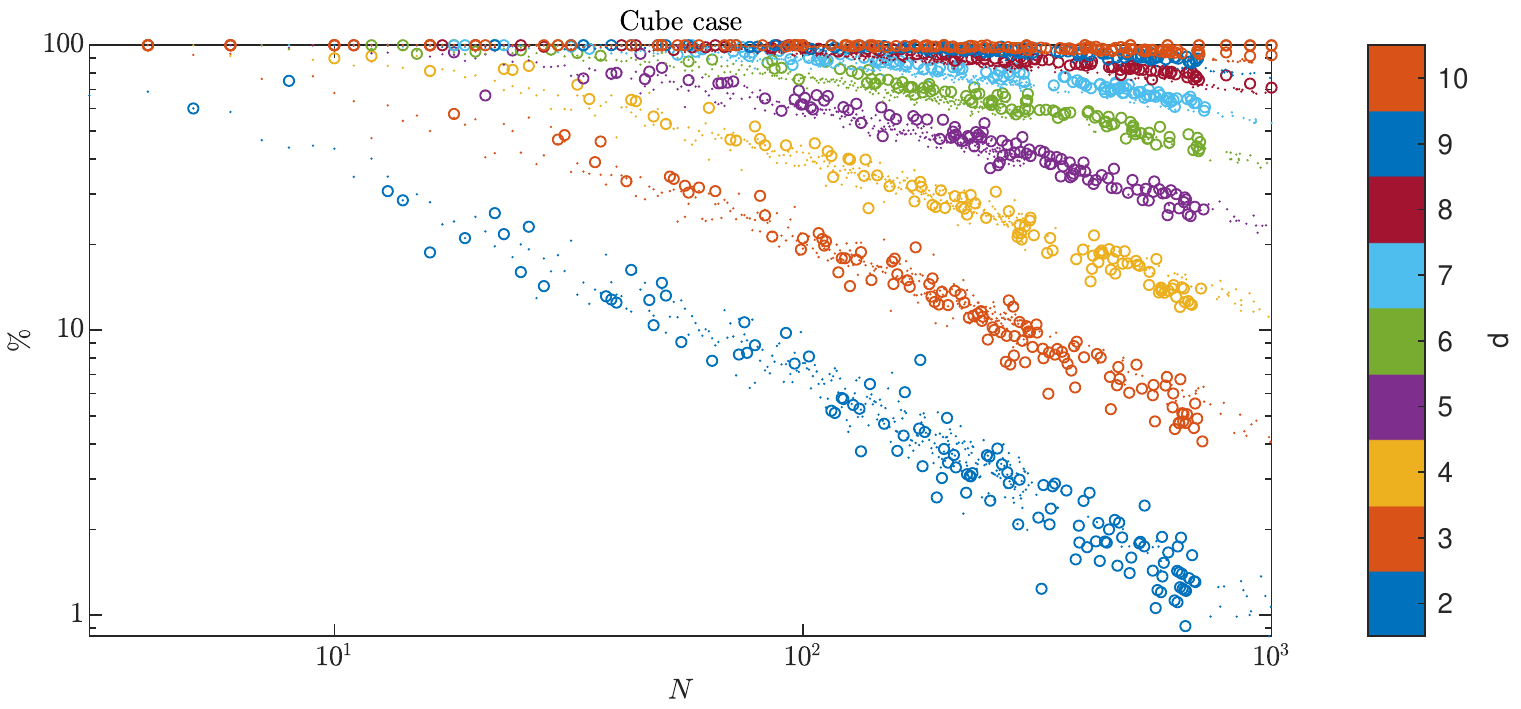}
 	\caption{\small 
     Fraction of points or ellipses which belong to the convex hull
     of randomly selected points or ellipses which are uniformly distributed 
     in the unit cube or have uniformly distributed real and imaginary part 
     in the unit cube, respectively.}
 	\label{fig_extremal_cube}
 \end{figure}

\subsubsection*{Gaussian sampled ellipsoids}

Also for points and ellipses with real and imaginary part 
sampled from a $d$-dimensional normal distribution, the behaviour is similar.
See Figure~\ref{fig_extremal_gaussian} for a visualization of the obtained numerical results.

\begin{figure}[!ht]
	\centering
	\includegraphics{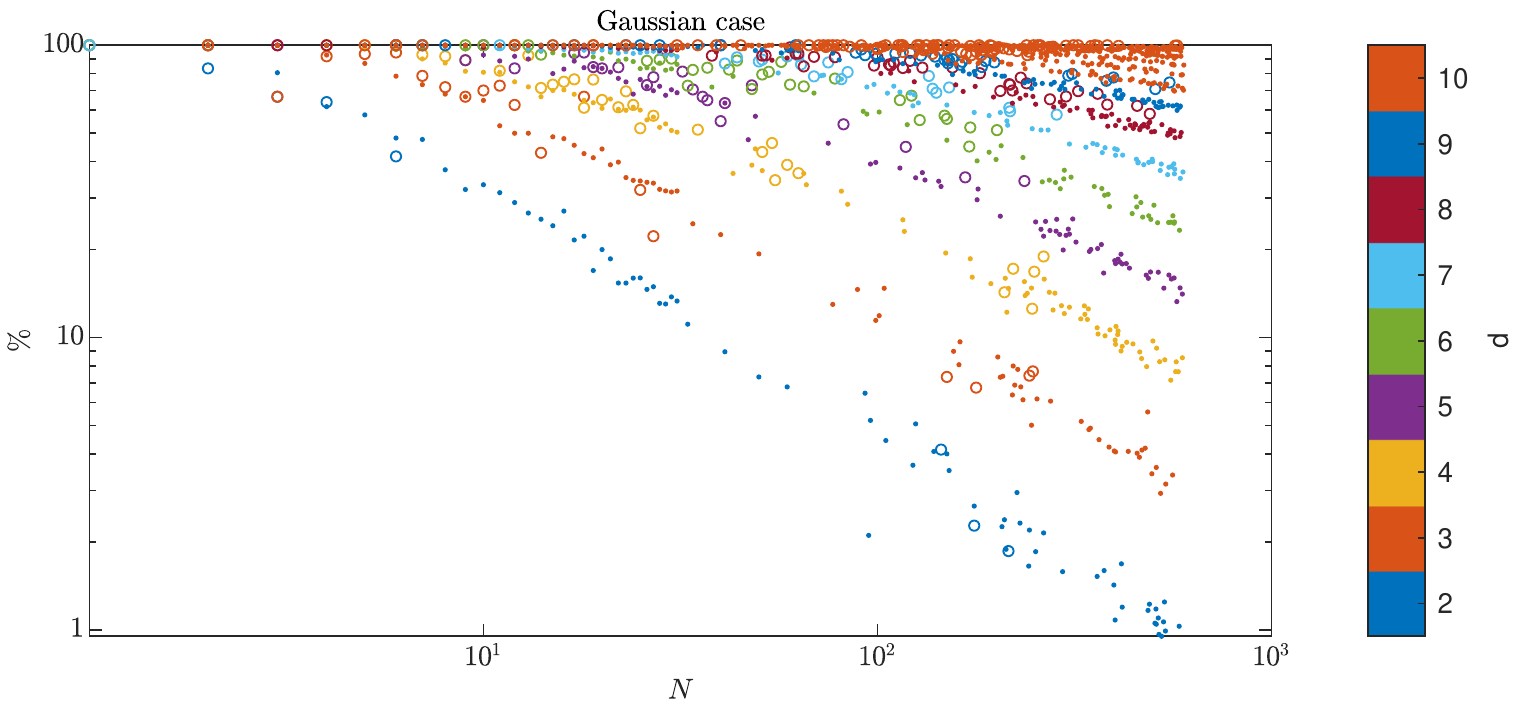}
		\caption{\small  
        Fraction of points or ellipses which belong to the convex hull 
        of randomly selected points or ellipses which are 
        normal distributed or have normal distributed real and imaginary part, respectively.}
	\label{fig_extremal_gaussian}
\end{figure}

\subsection{Lyapunov function for a discrete time linear system}

\begin{figure}[!ht]
	\centering
	\includegraphics{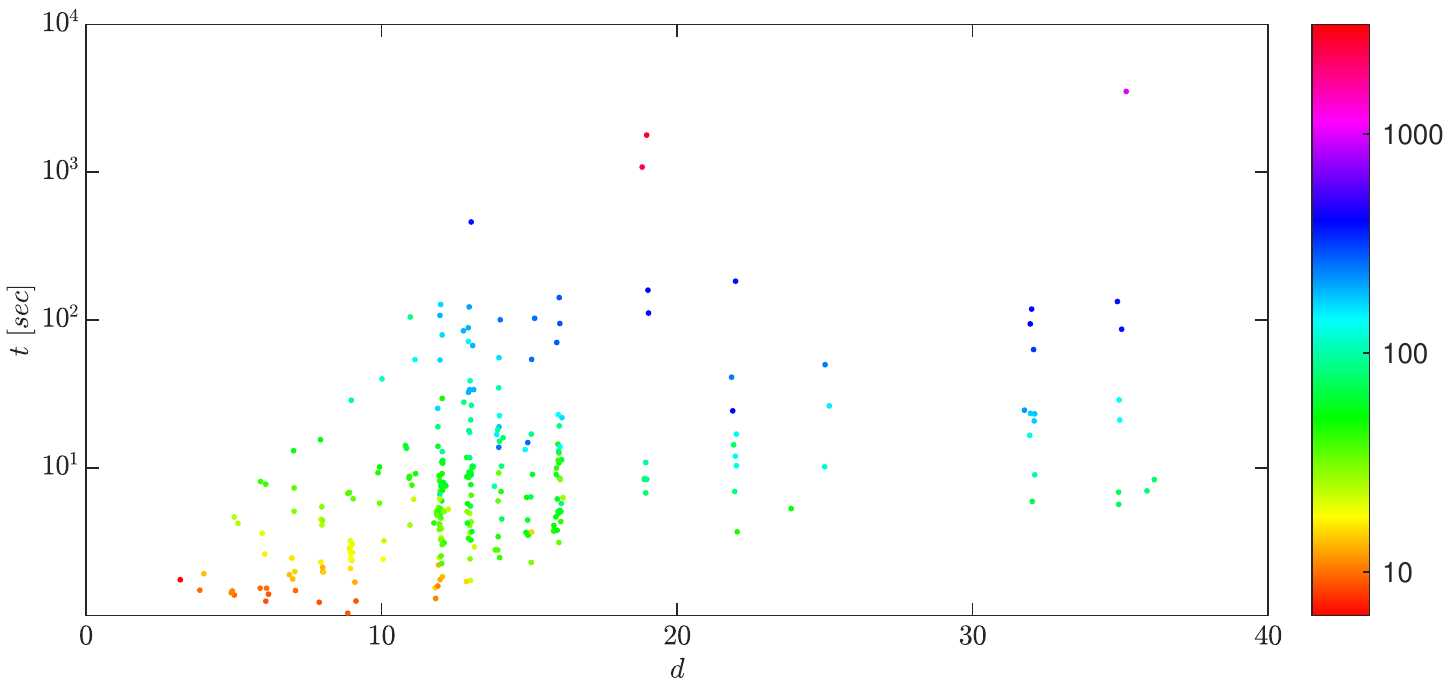}
	\caption{\small 
    The computation time for evaluating an invariant elliptic polytope $P$ 
    for a given matrix $A$ with complex leading eigenvalue such that 
    $AP\subset\rho(A)P$ holds. 
    The colour indicates the number of vertices of the polytope.}
	\label{fig_lyapunov}
\end{figure}	

Given a linear system defined by a $d\times  d$ matrix $A$ 
with a complex leading eigenvalue, which is supposed to be unique and simple.
We need to construct a norm~$\|\cdot\|$ in~$\re^d$ such that 
$\|A\bx\| \, \le \, \rho(A)\|\bx\|$ for all~$\bx \in \re^d$. 
This is the same as constructing a symmetric convex body~$P \subset \re^d$
for which  $AP\subset\rho(A)P$. It is obtained as an elliptic polytope by an 
iteration method,  
 see Section~\ref{sec_introduction}, Application~2. 
In Figure~\ref{fig_lyapunov} the time needed to compute the invariant elliptic polytope $P$
is plotted against the dimension. 
The colour indicates the number of vertices of the set~$V$.

\subsection{Invariant polytope algorithm}

\begin{figure}[!ht]
	\centering
	\includegraphics{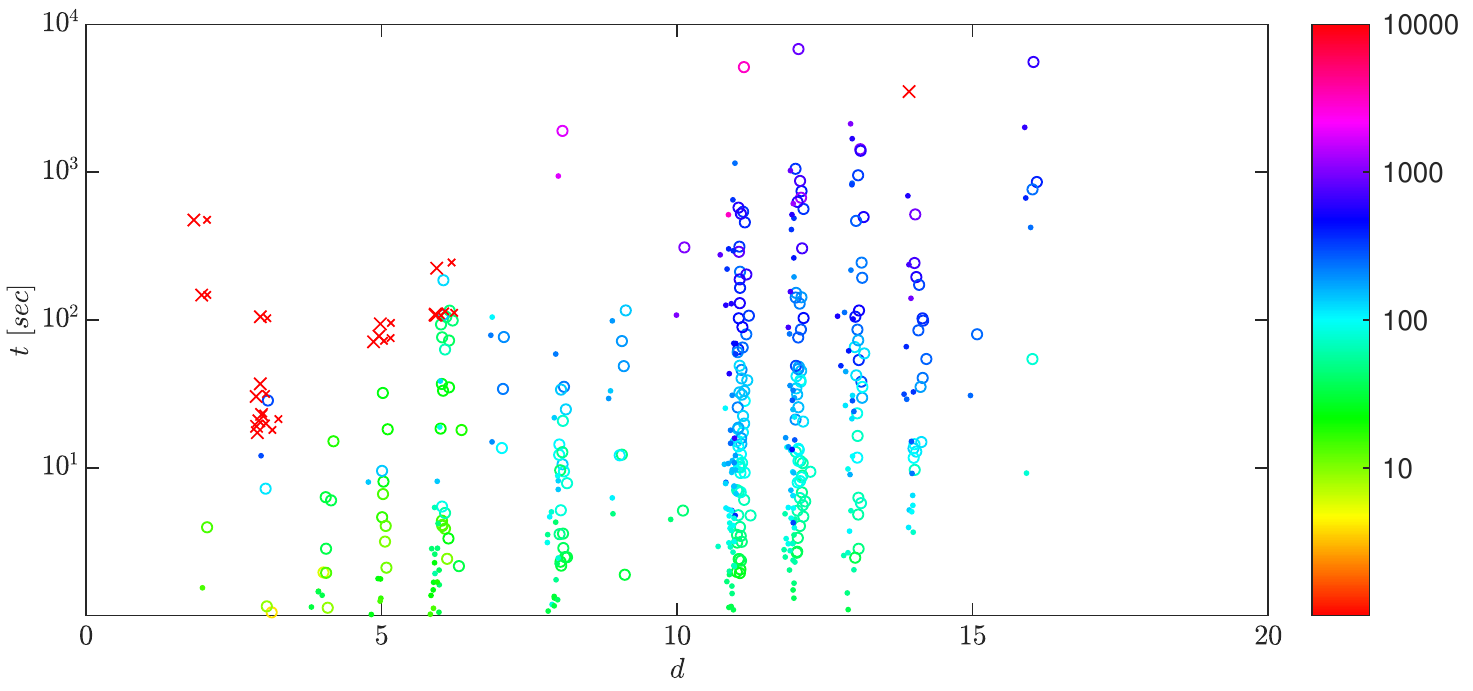}
	\caption{\small 
    The time of computation of the joint spectral radius 
    for a pair of matrices $A_1,A_2\in\re^d$,
    whose spectrum maximizing product has complex leading eigenvalue. 
    The colour indicates the number of vertices of the polytope.
    }
	\label{fig_ipa}
\end{figure}

Now we analyse the performance of the Invariant polytope algorithm 
for computation of the joint spectral radius of a set of matrices.
The elliptic polytopes are applied in the case when the spectrum maximizing product has a 
complex leading eigenvalue. In the construction of the 
invariant elliptic polytopes we can use each of our methods. 
The  numerical tests show that the projection method 
always performs better than the corner cutting method and
that the mixed method always performs better than the projection method.
Thus, only two significant algorithms remain, 
the complex polytope method and the projection method. We are comparing them.

In Figure~\ref{fig_ipa} the results of our experiments are plotted.
On the x-axis we have  the dimension of the generated example, 
the y-axis shows the time needed to compute an invariant polytope.
The x-values are slightly distorted for better readability.
Similar to the case of random matrices with real leading eigenvalue, 
it seems that the existence of a spectrum maximizing product with finite length is generic.
The results obtained using the complex-polytope method are marked with a $\cdot$ symbol,
the results obtained using the projection method are marked with a $\circ$ symbol.
Examples where the algorithm could not find an invariant polytope
are marked in both cases with a red $\color{red}\times$ symbol.
We suspect the reason why the Invariant polytope algorithm 
may not terminate within reasonable time for certain examples is
a long spectrum maximizing product for the set under test.

\begin{rmk}\label{r.110}
In practice it occurs rather seldom that 
a spectrum maximizing product 
of a set of matrices possesses a complex leading eigenvalue 
and whose length is greater than one. 
One such example is given in the Appendix, Example~\ref{ex.110}.
\end{rmk}

\newpage

\appendix

\section{Appendix}

\subsection*{Proof of Theorem~\ref{th.8}.}
We begin with the following technical fact. 
Let us have a vector $\ba\in \re^2$ and a line~$\ell$ on~$\re^2$
which is not parallel to~$\ba$. An {\em affine symmetry} 
about~$\ell$ 
along $\ba$ is an affine transform 
that for each $\bx \in \ell$ and $t \in \re$, 
maps the point $\bx + t\ba$ to $\bx - t\ba$.  
 If $\ba \perp \ell$, then this is the usual  (orthogonal) symmetry. 
 \begin{lma}\label{l.60}
 Let $O$ be an arbitrary point on the side of a convex polygon
 different from its midpoint. 
 Then there exists an affine symmetry about this side arbitrarily close 
   to an orthogonal symmetry such that the distances from $O$
   to the images of the vertices of this polygon are all different.  
 \end{lma}
\begin{proof}
 If we choose the origin at~$O$ and one of the basis vectors 
 along that side, then the matrix of an arbitrary affine symmetry is 
 $$
S \ = \ 
\left(
\begin{array}{rr}
1 & a\\
0 & -1
\end{array}
\right)\, ,                                                           
$$
where $a$ is an arbitrary number. If the images $A\bx$ and $A\by$
of two vertices~$\bx \ne \by$  are equidistant from~$O$, then the vectors~$A(\bx + \by)$
and $A(\bx - \by)$ are orthogonal and hence $(\bx - \by)A^TA (\bx + \by) = 0$. 
This is a quadratic equation in~$a$, which has at most two solutions. 
Hence, there  exists only a finite number of values of~$a$ for which 
some of images of vertices are equidistant from~$O$.  
\end{proof}

  \begin{prop}\label{p.30}
For every~$n\ge 2$ and $\varepsilon > 0$, there exists a polyhedron~$Q_n$ in~$\re^{n+2}$
with at most~$2n+3$ facets 
whose orthogonal projection to some two-dimensional plane 
is a $2^n$-gon  such that: 
1) Its distance  (in the Hausdorff metric) to a regular $2^n$-gon
centred at the origin is less than~$\varepsilon$.
2) The distances from its~$2^n$ 
vertices to the origin are all different.  
 \end{prop}
\begin{proof}
Applying the construction~\eqref{eq.rTn} for~$r= 1$, we obtain 
 a polyhedron that consists of points~$(x_1, \ldots , x_{2n+2})^T \in \re^{2n+2}$
 satisfying the system~\eqref{eq.rTn}. 
 That system contains $n$ linear equations and $2n+3$ linear inequalities. 
Hence, it defines an $(n+2)$-dimensional polyhedron with at most~$2n+3$ facets. 
Its   projection to the plane~$(x_{2n+1}, x_{2n+2})$ is a regular 
 $2^n$-gon. 
 Now, in each  iteration $j = 1, \dots, n$ of the construction~\eqref{eq.rTn}, 
 we replace the symmetry about the line~$\ell_{\alpha_{n-j+1}}$ 
 by a close affine symmetry about the same line. 
 Invoking Lemma~\ref{l.60} we can choose this symmetry so that 
 the resulting polygon  has all its vertices on different distances 
 from the origin. Hence, the polygon obtained after the last iteration also possesses 
 this property. 
\end{proof}

\begin{proof}[Proof of Theorem~\ref{th.8}.]
After applying Proposition~\ref{p.30} 
for $n=N-2$, we obtain a polyhedron~$Q_{N-2} \subset \re^{N}$
whose two-dimensional projection to the plane~$(x_{2N-3}, x_{2N-2})$ 
is a~$2^{N-2}$-gon close to a regular~$2^{N-2}$-gon. 
Then for the quadratic form~$x_{2N-3}^2 + x_{2N-2}^2$, each vertex of this 
polygon is a local maximum and all the values in those points are different. 
\end{proof}

\subsection*{Set of matrices with spectral maximizing product of length 2}

\begin{ex}\label{ex.110}
{\em For $\alpha,\beta\in(-\pi/2,\pi/2)$, $\alpha\ne\beta$, the set $\{T_0,T_1\}$, 
\begin{equation*}
T_0 = 
\begin{pmatrix} 
0 & 0 & 0\ \\ -\sin\, \alpha & \cos\, \alpha & 0\ \\ 
\phantom{-}\cos\, \alpha & \sin \, \alpha & 0\ 
\end{pmatrix},
\quad
T_1 =
\begin{pmatrix}
\ 0 & -\sin\, \beta & \cos\, \beta\ \\ \ 0 & \phantom{-}\cos\, \beta & 
\sin\, \beta\ \\ \ 0 & 0 & 0\ 
\end{pmatrix},
\end{equation*}
has $T_0T_1$ as spectrum maximizing product,
i.e.\ up to permutations and powers the normalized spectral radius 
of all other products of matrices $T_0$ and $T_1$ 
is strictly less than $\rho(T_0T_1)^{1/2}=1$.}
\end{ex}

\newcommand{\doi}[1]{\href{https://doi.org/#1}{doi: #1}}
\newcommand{\arxiv}[1]{\href{https://arxiv.org/abs/#1}{arXiv: #1}}

 \end{document}